\documentclass[leqno,11pt,a4paper]{amsart}
\usepackage{bw_jc}

\newcommand{\cgh}{\mathfrak{c}^{\text{\textnormal{GH}}}}

\newcommand{\btikz}[2]{\begin{figure}[h]\caption{#1}
\vspace{3pt}
\label{#2}\begin{center}\begin{tikzpicture}}
\newcommand{\etikz}{\end{tikzpicture}\end{center}\end{figure}}
\newcommand{\uk}{\mfk{u}}
\newcommand{\lk}{\mfk{l}}
\graphicspath{{images/}}
\begin{document}
\title[Lattice Formulas For Rational SFT Capacities]{Lattice Formulas For Rational SFT Capacities}
\author{J.~Chaidez}
\address{Department of Mathematics\\University of California, Berkeley\\Berkeley, CA\\94720\\USA}
\email{jchaidez@berkeley.edu}
\vs
\maketitle

\begin{abstract} We initiate the study of the rational SFT capacities of Siegel using tools in toric algebraic geometry. In particular, we derive new (often sharp) bounds for the RSFT capacities of a strongly convex toric domain in dimension $4$. These bounds admit descriptions in terms of both lattice optimization and (toric) algebraic geometry. Applications include (a) an extremely simple lattice formula for for many RSFT capacities of a large class of convex toric domains, (b) new computations of the Gromov width of a class of product symplectic manifolds and (c) an asymptotics law for the RSFT capacities of all strongly convex toric domains. 
\end{abstract}

\section{Introduction} A symplectic capacity $\mathfrak{c}$ is, roughly speaking, a numerical invariant assigning a number $\mathfrak{c}(X)$ to each symplectic manifold $X$ (possibly of a restricted type). The main axiom of symplectic capacities is that
\begin{equation}
\mathfrak{c}(X) \le \mathfrak{c}(X') \qquad\text{if}\qquad X \text{ symplectically embeds into }X'
\end{equation}
Here a symplectic embedding of symplectic manifolds $X \to X'$ is a smooth, codimension $0$ embedding $\varphi:X \to X'$ that intertwines the symplectic forms. The prototypical example of a capacity is the \emph{Gromov width} $\mathfrak{c}_{\on{Gr}}$, defined as follows.
\begin{equation} \label{eqn:def_Gromov_width}
\mathfrak{c}_{\on{Gr}}(X) := \on{sup}\big\{\pi r^2 \; : \; B^{2n}(r) \text{ symplectically embeds into }X\big\}\end{equation}

\vspace{3pt}

Over the last several decades, various symplectic capacities (and their close relatives, spectral invariants) have taken on an increasingly prominent role in symplectic geometry. They have found dramatic applications to symplectic embedding problems (cf. \cite{hut_qua_11,mcd_sym_09,mcd_hof_11,cri_sym_14,cm2018,chs_hig_21}), as well as to questions in Hamiltonian and Reeb dynamics (cf. \cite{i2015,s2019}).

\subsection{RSFT capacities} \label{subsec:intro_RSFT_capacities}

\vspace{3pt}

Sympectic field theory (SFT) is a Floer homology framework for contact manifolds and symplectic cobordisms introduced by Eliashberg--Givental--Hofer \cite{egh2000}. Variants of SFT are built using chain groups generated (in some way) by closed Reeb orbits and whose differentials count (pseudo)holomorphic curves in symplectizations. Symplectic cobordisms induce maps of chain groups, well-defined up to homotopy. 

\vspace{3pt}

In \cite{sie_hig_19} Siegel introduced a new family of capacities of Liouville domains constructed using rational symplectic field theory (RSFT). Rational SFT is a variant of symplectic field theory that counts only genus $0$ holomorphic curves. Each \emph{RSFT capacity}, denoted by 
\begin{equation}\mathfrak{r}_P(X) \quad \text{for a Liouville domain}\quad X\end{equation}
is indexed by a tangency constraint $P$. A \emph{tangency constraint} at $m$ points specifies a way that a holomorphic curve can pass through a set of $m$ points, possibly with some order of tagency to a local divisor at each point. Every tangency constraint $P$ has an associated \emph{codimension} $\on{codim}(P)$. This is the amount that imposing that $P$ reduces the dimension of a moduli space of curves. A detailed discussion of tangency constraints and RSFT capacities appears in \S \ref{sec:RSFT_capacities}.

\vspace{3pt}

Roughly speaking, the RSFT capacity $\mathfrak{r}_P(X)$ is the area of a punctured, disconnected, genus $0$ holomorphic curve $C$ in $X$ of Fredholm index $\on{codim}(P)$ that satisfies the tangency constraint $P$ at a fixed set of points in $X$.

\begin{figure}[h]
\centering
\caption{A cartoon of a disconnected curve $C = C_1 \cup C_2 \subset X$ (in green) satisfying a tangency constraint at $4$ points (in blue) and tangent to divisors $D_i$ (in red).}
\includegraphics[width=.8\textwidth]{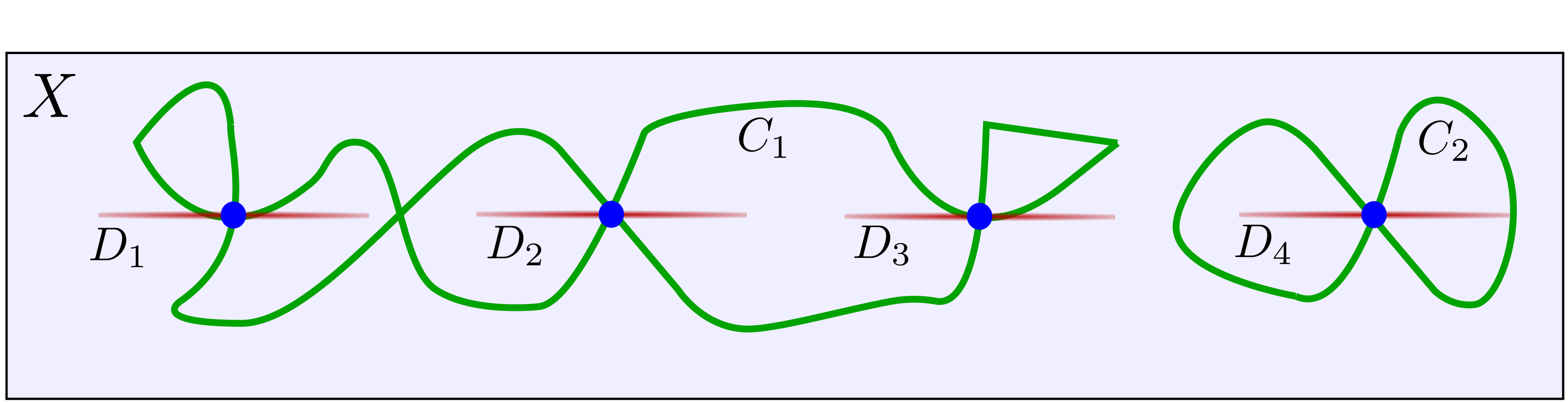}
\label{fig:tangency_condition_intro}
\end{figure}

\subsection{Toric domains} \label{subsec:intro_toric_domains} In general, symplectic capacities are difficult to compute. However, a number of elegant formulas are known in the toric setting. 

\begin{definition} A \emph{toric domain} $X_\Omega \subset \C^n$ is a domain given by the inverse image of a \emph{moment domain} $\Omega \subset [0,\infty)^n$ under the moment map
\begin{equation}
\mu:\C^n \to \R^n \qquad \text{with}\qquad \mu(z_1,\dots,z_n) = (\pi |z_1|^2,\dots,\pi |z_n|^2)
\end{equation}
\end{definition}

Toric domains are relatively accessible objects of study in symplectic geometry since the Reeb dynamics on their contact boundaries and their holomorphic curves are very well understood. The symplectic capacities of toric domains have been studied extensively. See the work of Lu \cite{lu2006}, Choi et al \cite{ccfhr_sym_14}, Landry et al \cite{lmt2015}, Cristofaro-Gardiner \cite{cri_sym_14}, Gutt--Hutchings \cite{gh_sym_18}, Gutt at al \cite{ghr2020} Hutchings \cite{hut_qua_11,hut_ech_19}, Siegel \cite{sie_com_19}, Wormleighton \cite{wor_ech_19,wor_tow_21} and Chaidez--Wormleighton \cite{cw_ech_20}.

\vspace{3pt}

Many of the results on capacities in the toric setting apply when $\Omega$ satisfies a convexity or concavity hypothesis. In this paper, we will study the following special type of domain.

\begin{definition} \label{def:strongly_convex} A moment domain $\Omega \subset [0,\infty)^n$ is \emph{strongly convex} if the domain
\begin{equation}
\widehat{\Omega} := \{(x_1,\dots,x_n) \; : \; (|x_1|,\dots,|x_n|) \in \Omega\} \qquad\text{is convex and compact}
\end{equation}
We say that $\Omega$ is \emph{weakly convex} if $\Omega$ is convex as a subset of $\R^n$ and contains a neighborhood of $0$. \end{definition} 

\noindent In \cite{gh_sym_18,hut_qua_11,wor_ech_19}, several formulas for the capacities of convex toric domains are given in terms of the optimum $\Omega$-norm of a (set of) lattice vectors in $\Z^n$ satisying a few restrictions.

\begin{definition} \label{def:omega_norm} The \emph{$\Omega$-norm} of a strongly convex domain $\Omega \subset \R^n$ is the function
\begin{equation}
\|\cdot\|_\Omega^*:\R^n \to [0,\infty) \qquad\text{with}\qquad \|v\|_\Omega^* \; := \; \sup_{u\in\Omega} \; \langle u,v\rangle
\end{equation}\end{definition}

\noindent For instance, the following lattice formula for the Gutt--Hutchings capacities of strongly convex domains is proven in \cite{gh_sym_18}.
\begin{equation}\label{eqn:Gutt_Hutchings} \mathfrak{c}_k^{\on{GH}}(X_\Omega) = \on{min}\big\{\|v\|_\Omega^* \; : \; v = (k_1,\dots,k_n) \in \Z^n \;\text{with}\; k_i \ge 0 \;\text{and}\; \sum_i k_i = k\big\}\end{equation}
A similar lattice formula for the ECH capacities of convex toric domains is proven in \cite{hut_qua_11} and stated more explicitly in \cite{wor_ech_19}.

\vspace{3pt}

When a convex moment domain $\Omega$ is a rational-sloped polytope it determines a fan $\Sigma$ and a closed (possibly singular) toric variety $Y_\Sigma$ equipped with a symplectic form $\omega_\Omega$ dual to an ample divisor $A_\Omega$. The interior of $X_\Omega$ symplectically embeds into $Y_\Sigma$. This will be discussed in \S \ref{subsec:toric_varieties}.

\vspace{3pt}

This paper is part of a program \cite{wor_ech_19,cw_ech_20,wor_alg_20,wor_tow_21} of the second author (and sometimes the first) to understand the relationship between the quantitative symplectic geometry of $X_\Omega$ and the algebraic geometry of $Y_\Sigma$.

\subsection{Main results} \label{subsec:main_results} The purpose of this paper is to provide bounds and in some cases exact formulas for many of the RSFT capacities of a convex toric domain in dimension $4$. These bounds are formulated as lattice optimization problems, in the spirit of \cite{gh_sym_18,hut_qua_11}. 

\vspace{3pt}

\subsubsection{Lower bound} Our first main result is the following lattice optimization lower bound for the RSFT capacities of a convex toric domain in any dimension. 
\begin{thm*}[Corollary \ref{cor:main_lower_bound}] \label{thm:main_lowerbound} Let $\Omega \subset \R^n$ be a weakly convex moment domain, and let
\[\mathfrak{l}_k(\Omega) = \underset{m \ge k+1}{\on{min}}\big\{\sum_{i=1}^m \|v_i\|_\Omega^* \; : \; v_i \in \Z^n \setminus 0 \;\text{satisfying}\; \sum_i v_i = 0 \big\}\]
Let $P$ be a tangency constraint of codimension $(2n - 2) \cdot k$. Then
\[\lk_k(\Omega) \leq \mfk{r}_P(X_\Omega)\]
\end{thm*}
The basic outline of the proof is as follows. By a simple limiting argument, we can assume that $\Omega$ is smooth and contained in $(0,\infty)^n$. In this case, we can identify $X_\Omega$ with the unit disk bundle in $T^*T^n$ with respect to a certain $T^n$-invariant Finsler norm on $T^n$.
\[D^*_\Omega T^n \subset T^*T^n\]
The Reeb orbits (i.e. Finsler geodesics) on the boundary $S^*_\Omega T^n$ live in $S^1$-Morse--Bott families. Each family corresponds to a (non-zero) integer vector $v \in \Z^n \setminus 0$ and the action of an orbit in the family is given by the $\Omega$-norm of $v$. 

\vspace{3pt} The model case is $\Omega = B^n$, when this is simply the standard picture of geodesics on flat $T^n$. See Figure \ref{fig:lower_bound_intro} for a depiction of this discussion.

\begin{figure}[h]
\centering
\caption{A cartoon of the correspondence between collections of lattice points and null-homologous sets of Finsler geodesics on $T^n$.}
\includegraphics[width=.9\textwidth]{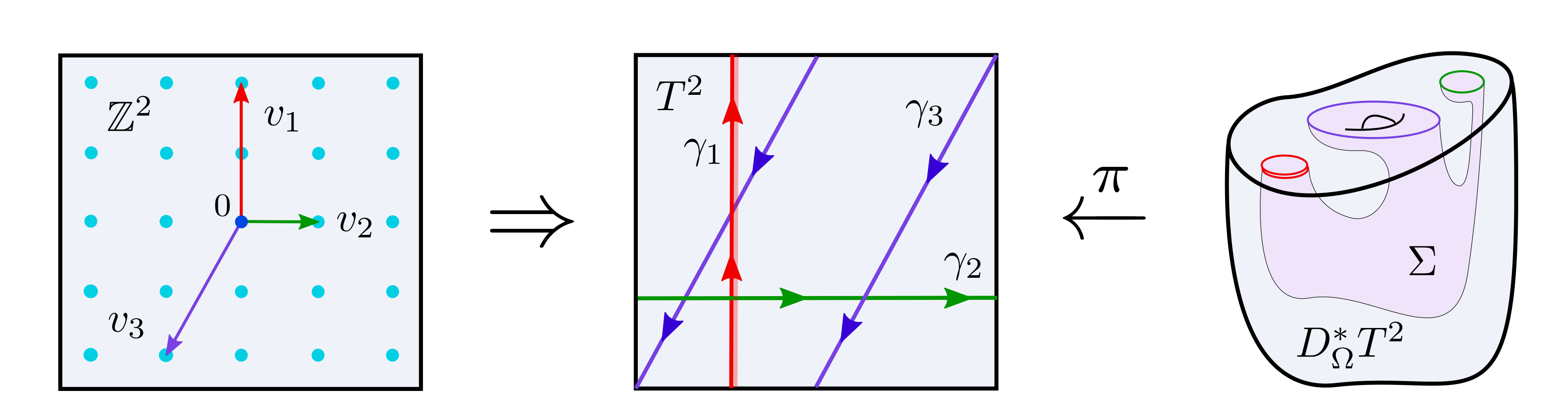}
\label{fig:lower_bound_intro}
\end{figure}

The RSFT capacity $\mathfrak{r}_P(D^*_\Omega T^n)$ is, intuitively, the area of a rigid, punctured holomorphic curve $C$ satisfying the tangency constraint $P$. The curve $C$ must have Fredholm index satisfying
\begin{equation} \label{eqn:intro_lb_proof_1} \on{ind}(C) = \on{codim}(P)\end{equation}
Otherwise, the virtual dimension of the moduli space of curves satisfying the constraint $P$ would not be $0$ near $C$, and thus $C$ would not be rigid. On the other hand, we can use a simple index calculation to prove that any such punctured curve must have
\begin{equation} \label{eqn:intro_lb_proof_2} \on{ind}(C) \le (2n-2) \cdot (\# \text{ punctures of }C - 1)\end{equation}
Combining (\ref{eqn:intro_lb_proof_1}) and (\ref{eqn:intro_lb_proof_2}), and using the fact that the area of $C$ is precisely the sum of the actions of the Reeb orbits that it bounds gives Theorem \ref{thm:main_lowerbound}. A detailed argument is given in \S \ref{sec:lower_bounds_via_RSFT}.

\begin{remark} This argument was inspired by a paper \cite[\S 3]{cm2018} of Cieliebak--Mohnke. \end{remark}

\subsubsection{Upper bound} Our second main result is a lattice upper bound for the RSFT capacities $\mathfrak{r}_P$ of a strongly convex toric domain in dimension $4$. It is similar to Theorem \ref{thm:main_lowerbound} with two key differences. 

\vspace{3pt}

First, the tangency constraints $P$ in Theorem \ref{thm:main_upperbound} are \emph{lax}. Informally, a tangency constraint $P$ at $m$ points $p_1,\dots,p_m \in X$ is lax if a immersed surface $\Sigma$ satisfying $P$  has only one branch running through each point $p_i$. Rephrased, only one marked point of $\Sigma$ passes through each point $p_i$. See \S \ref{subsec:tangency_constraints} and Definition \ref{def:laxness} for a more precise description.

\vspace{3pt}

Second, we optimize over a slightly different set of lattice vector sequences. Namely, we will only consider sequences $v_1,\dots,v_m$ satisfying the following conditions.
\begin{equation}\label{eqn:upperbound_condition_intro}
\text{if $v_i \in \text{span}(e_j)$ for each $i$ and a fixed $j \in \{0,1\}$, then $m = 2$, $v_1 = \pm e_j$ and $v_2 = \pm e_j$.}
\end{equation}

\begin{thm*}[Theorem \ref{thm:main_upperbound_body}] \label{thm:main_upperbound} Let $\Omega \subset \R^2$ be a strongly convex moment domain and let
\[\mathfrak{u}_k(\Omega) = \underset{m \ge k+1}{\on{min}}\big\{\sum_{i=1}^m \|v_i\|_\Omega^* \; : \; v_i \in \Z^2 \setminus 0 \;\text{satisfying}\; \sum_i v_i = 0 \; \text{and}\; (\ref{eqn:upperbound_condition_intro}) \big\}\]
Let $P$ be a lax tangency constraint of codimension $2k$. Then
\[\mfk{r}_P(X_\Omega) \leq \uk_k(\Omega)\]
\end{thm*}

Despite the similarity of the upper bound in Theorem \ref{thm:main_upperbound} to the lower bound in Theorem \ref{thm:main_lowerbound}, the proof is entirely different. It factors through a connection to toric algebraic geometry. We now outline this argument.

\vspace{3pt}

First assume (via a limiting argument) that $\Omega$ is a rational polytope. In this situation, $\Omega$ determines a closed (possibly singular) toric variety $Y_\Sigma$ associated to a fan $\Sigma$. The variety $Y_\Sigma$ comes with a natural symplectic form $\omega_\Omega$ determined by $\Omega$ and the interior of $X_\Omega$ embeds into $Y_\Sigma$ as the complement of a divisor $A_\Omega$ in $Y_\Sigma$ that is dual to $\omega_\Omega$. We can reformulate $\mathfrak{l}_k$ and $\mathfrak{u}_k$ as minima over sequences $\bar{v}$ of vectors $v_i$, each of which generates a $1$-dimensional cone in the fan $\Sigma$. A sequence of this type precisely specifies an effective, movable curve class $C$ in the toric variety $Y_\Sigma$. Moreover, the number of vectors is related to the Chern class of $Y_\Sigma$ evaluated on $C$. 

\begin{figure}[h]
\centering
\caption{An cartoon of the correspondence between sequences $\bar{v}$ of fan vectors and curves $C$ in $Y_\Sigma$. Each ray $\rho_i$ in $\Sigma$ yields a divisor $D_i$ in $Y_\Sigma$, and the intersections of $C$ with $D_i$ are governed by the number of vectors in $\bar{v}$ in the ray $\rho_i$.}
\includegraphics[width=.9\textwidth]{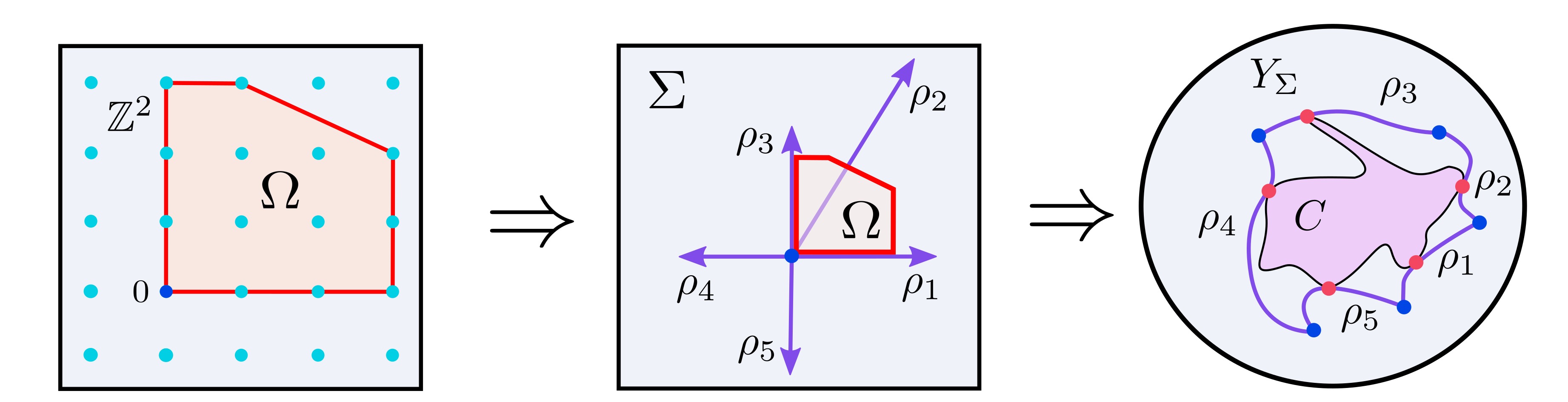}
\label{fig:upper_bound_intro}
\end{figure}

Inspired by the above intepretation of $\bar{v}$, we turn our attention to closed curves in $Y_\Sigma$. Now, after a small shrinking $X_\Omega$ symplectically embeds into $(Y_\Sigma,\omega_\Omega)$. Furthermore, a basic property of the RSFT capacities implies that $\mathfrak{r}_P(X_\Omega)$ is bounded by the area of any curve class $C$ that has non-vanishing Gromov--Witten invariants $\on{GW}_P(Y_\Sigma,C)$ with tangency condition $P$. The invariant $\on{GW}_P$ was introduced by McDuff--Siegel \cite{ms_cou_19} and is discussed in \S \ref{subsec:gromov_witten_invariants}. 

\vspace{3pt}

Thus, we must show that the curve class $C$ associated to vector $\bar{v}$ has non-vanishing Gromov--Witten invariant $\on{GW}_P$. When $P$ is lax, a result of McDuff--Siegel (see Lemma \ref{lem:immered_spheres}) implies further that it suffices to find an immersed symplectic sphere representing $C$ in homology. Such representatives can be constructed geometrically from certain singular rational curves, called \emph{cocharacter curves} (see \S \ref{subsec:cocharacter_curves}) when $\Omega$ is strictly convex. This proof is carried out in detail in \S \ref{subsec:upper_bounds_via_movable_curves}.

\begin{remark} The strange condition (\ref{eqn:upperbound_condition_intro}) is disappointing, since $\mathfrak{l}_k$ and $\mathfrak{u}_k$ would coincide in an ideal world. In fact, (\ref{eqn:upperbound_condition_intro}) arises naturally in the context of our proof.

\vspace{3pt}

The curve classes ruled out by (\ref{eqn:upperbound_condition_intro}) are precisely the higher multiples $k \cdot F$ of a self-intersection $0$ sphere $F$. A simple argument with adjunction shows that $k \cdot F$ can only be represented by an immersed symplectic sphere with positive self-intersections when $k = 1$. Thus the proof of Theorem \ref{thm:main_upperbound} only works in that case. This observation was made in \cite[p. 65, Ex 5.1.4]{ms_cou_19} for $\P^1 \times \P^1$.
\end{remark}

\begin{remark} The lower bound $\mfk{l}_k$ can also be formulated (very naturally) in algebro-geometric terms for polarized varieties $(Y,A)$ of any dimension. We do this in \S\ref{subsec:algebraic_formulation}.
\end{remark}

\subsubsection{Closed obstructions} Our proof of Theorem \ref{thm:main_upperbound} implies an estimate for the RSFT symplectic capacities of \emph{any} star-shaped domain that embeds into a toric surface. 

\begin{thm*}[Theorem \ref{thm:main_upperbound_closed_body}] \label{thm:main_upperbound_closed} Let $X \to Y_\Sigma$ be a symplectic embedding of a star-shaped domain $X$ into a strongly convex toric surface $Y_\Sigma$ with symplectic form $\omega_\Omega$. Then
\begin{equation} \label{eqn:main_upperbound_closed}
\mfk{r}_P(X) \leq \uk_k(\Omega) \quad\text{where}\quad \text{$P$ is lax and } \on{codim}(P) = 2k
\end{equation}
\end{thm*}

There is also a stable version of Theorem \ref{thm:main_upperbound_closed} for certain tangency conditions. This result can be used to obstruct embeddings into higher dimensional manifolds. For convenience, let $P(k)$ denote the $1$-point tangency condition for a surface passing through a single point $p$ with tangency order $k+1$ at a divisor $D$ through $p$. 

\begin{thm*}[Corollary \ref{cor:main_upperbound_stable_body}] \label{thm:main_upperbound_stable} Let $X \to Y_\Sigma \times Z$ be a symplectic embedding of a star-shaped domain $X$ into the product of a strongly convex toric surface $Y_\Sigma$ and a closed symplectic manifold $Z$. Then
\begin{equation} \label{eqn:main_upperbound_stable}
\mfk{r}_{P(k)}(X) \leq \uk_{k+1}(\Omega) 
\end{equation}
\end{thm*}

\subsection{Applications} We now discuss several applications of the main results to symplectic embedding problems and computation problems for symplectic capacities.

\subsubsection{Calculations} We can associate widths $a(\Omega)$ and $b(\Omega)$ to a convex domain $\Omega \subset \R^2$ as so.
\begin{equation}a(\Omega) := \max{a \; : \; (a,0) \in \Omega} \qquad b(\Omega) := \max{b \; : \; (0,b) \in \Omega}\end{equation}
The bounding quantities $\mathfrak{l}_k$ and $\mathfrak{u}_k$ agree when these two widths coincide.
\begin{prop*}[Lemma \ref{lem:lk_uk_equal_when_a_is_b}] If $\Omega$ is strongly convex and $a(\Omega) = b(\Omega)$ then $\mathfrak{l}_k(\Omega) = \mathfrak{u}_k(\Omega)$ for all $k$. \end{prop*}
\noindent As a consequence, we have the following extremely simple closed form for many RSFT capacities.
\begin{cor*} \label{cor:simple_lattice_formula} If $\Omega \subset \R^2$ is strongly convex with $a(\Omega) = b(\Omega)$, and $P$ is a lax tangency constraint, then
\[
\mfk{r}_P(X_\Omega) = \lk_k(\Omega) \quad\text{where}\quad  \on{codim}(P) = 2k
\]
\end{cor*}
\noindent This formula is applicable to cubes, balls and many other convex domains.

\begin{example} \label{ex:lk_computation} Consider the strongly convex region depicted below, with parameter $r\in\Z_{\geq1}$.
\begin{figure}[h]
\centering
\caption{The strongly convex region $\Omega$}

\vspace{9pt}
\includegraphics[width=.2\textwidth]{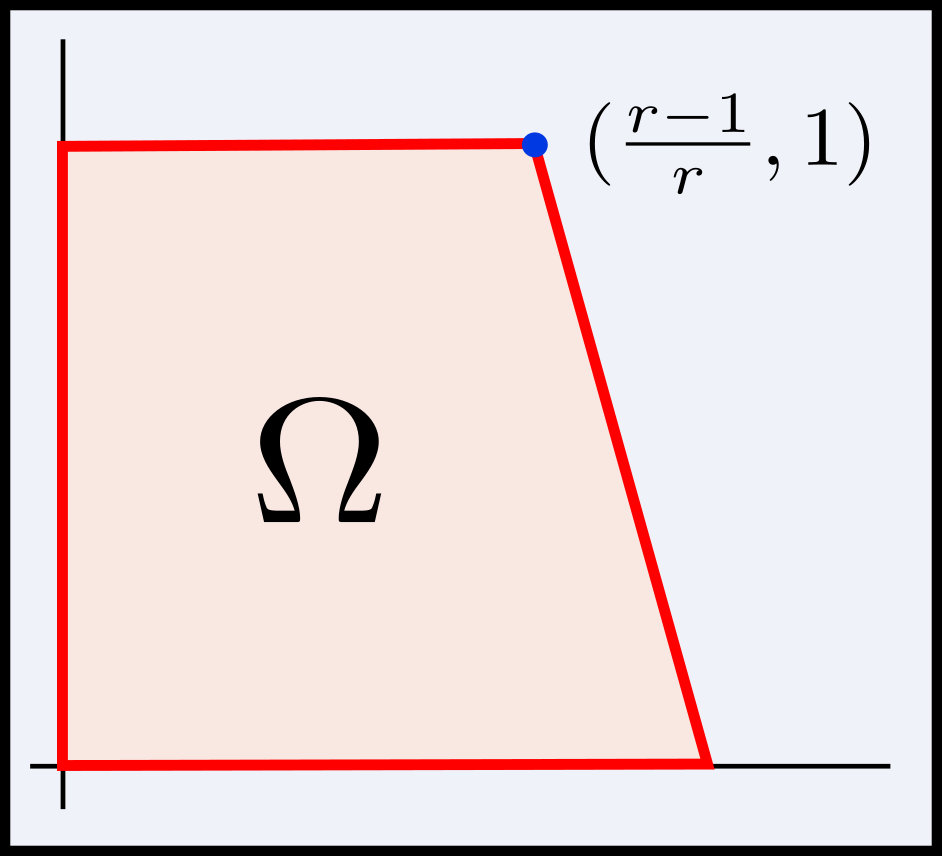}
\label{fig:examin_convex_domain}
\end{figure}

\noindent A domain $\Omega$ of this form satisfies $a(\Omega)=b(\Omega)$ and so we are able to compute the RSFT capacities with lax tangency constraints by applying Corollary \ref{cor:simple_lattice_formula}. We compute that
\begin{equation} \mathfrak{r}_P(X_\Omega) = \mfk{l}_k(\Omega)=\begin{cases}
\tfrac{k+1}{2} & \text{$k = \on{codim}(P)/2$ is odd} \\
\tfrac{k-2}{2}+\frac{2r-1}{r} & \text{$k = \on{codim}(P)/2$ is even}
\end{cases} \end{equation}
We can apply (\ref{eqn:Gutt_Hutchings}) for the Gutt--Hutchings capacities and the lattice formula for ECH capacities \cite[Cor.~A.12]{cri_sym_19} to compare these two capacities with the $\mfk{l}_k$. See Table \ref{table:comparison} for a comparison of the first $5$ capacities when $r = 4$. 

\begin{table}[!h]
\caption{Comparing $\mfk{l}_k$ to $\cgh_k$ and $\cech_k$}
\begin{tabular}{c|c|c|c}
$k$ & $\mfk{l}_k = \mathfrak{r}_P$ & $\cgh_k$ & $\cech_k$ \\ 
\hline
1 & 1 & 1 & 1 \\  
2 & $\tfrac{7}{4}$ & $\tfrac{7}{4}$ & $\frac{7}{4}$ \\
3 & 2 & $\tfrac{5}{2}$ & 2 \\
4 & $\tfrac{11}{4}$ & $\tfrac{13}{4}$ & $\frac{11}{4}$ \\
5 & 3 & 4 & 3
\end{tabular}
\label{table:comparison}
\end{table}
\end{example}

\begin{remark}[Toric Computation] In \S \ref{subsec:fan_formulation} we introduce a formulation of $\lk_k(\Omega)$ in terms of a toric resolution of singularities, or a refinement $\widetilde{\Sigma}$ of the fan $\Sigma$ supporting $\Omega$. It is often advantageous to use this toric reformulation in calculations, and we used it when computing Example \ref{ex:lk_computation}. \end{remark}

\subsubsection{Strong Viterbo} Recall that a symplectic capacity $\mathfrak{c}$ is called \emph{normalized} if it satisfies
\[\mathfrak{c}(B^{2n}) = \mathfrak{c}(B^2 \times \C^{n-1}) = \pi\]
A well-known conjecture of Viterbo states that all normalized capacities agree on convex sets.
\begin{conjecture*}[Viterbo] If $\mathfrak{c}$ is a normalized capacity and $X \subset \C^n$ is convex, then $\mathfrak{c}(X) = \mathfrak{c}_{\on{Gr}}(X)$. 
\end{conjecture*}

A toric manifold $X$ is convex as a subset of $\C^n$ if and only if it is strongly convex \cite[Prop 2.3]{ghr2020}. In this setting the Viterbo conjecture have been verified \cite[Thm 1.7]{ghr2020}, and any normalized capacity is given by
\begin{equation} \mathfrak{c}_{\on{Gr}}(X) = \on{min}(a(\Omega),b(\Omega))\end{equation}
We show that $\mathfrak{l}_1(\Omega)$ and $\mathfrak{u}_1(\Omega)$ are both given by $\on{min}(a(\Omega),b(\Omega))$ (see Lemma \ref{lem:first_bound}). This independently verifies the toric version of the Viterbo conjecture for the normalized RSFT capacity $\mathfrak{r}_{\on{pt}}$ corresponding the constraint $\on{pt}$ of passing through a point.

\begin{prop*} \label{prop:Gromov_width_and_rP_agree} Let $X$ be a (strongly) convex toric domain. Then $\mathfrak{r}_{\on{pt}}(X) = \mathfrak{c}_{\on{Gr}}(X)$.
\end{prop*}

\subsubsection{Gromov width} We can use Proposition \ref{prop:Gromov_width_and_rP_agree} to compute the Gromov width of a large number of products of closed symplectic manifolds. 

\begin{prop*} Let $Y$ be a strongly convex toric symplectic $4$-manifold with moment polytope $\Omega$ and let $Z$ be a closed, semi-positive symplectic manifold with Gromov width larger than that of $Y$. Then
\[
\mathfrak{c}_{\on{Gr}}(Y \times Z) = \mathfrak{c}_{\on{Gr}}(Y) = \on{min}(a(\Omega),b(\Omega))
\] 
\end{prop*}

\begin{proof} First, fix radii $r > 0$ such that $\pi r^2 < \mathfrak{c}_{\on{Gr}}(Y)$. Then we have a symplectic embedding
\[
B^{4 + 2n}(r) \to B^{4}(r) \times B^{2n}(r) \to Y \times Z \qquad\text{where}\qquad \on{dim}(Y) = 2n
\]
Taking the limit as $\pi r^2$ approaches $\mathfrak{c}_{\on{Gr}}(Y) = \on{min}(a(\Omega),b(\Omega))$ yields the desired lower bound of $\mathfrak{c}_{\on{Gr}}(Y \times Z)$. Conversely, take any symplectic embedding $B^{4 + 2n}(r) \to Y \times Z$. Then by Theorem \ref{thm:main_upperbound_stable}
\[
\pi r^2 = \mathfrak{l}_1(\Delta(r))\le \mathfrak{r}_{\on{pt}}(B^{4 + 2n}(r)) \le \mathfrak{r}_{\on{pt}}(Y \times Z) \le \mathfrak{u}_1(\Omega) = \on{min}(a(\Omega),b(\Omega))
\]
Here $\Delta(r)$ is the moment polytope of $B^{4 + 2n}(r)$ in $\R^{2 + n}$.
\end{proof}

\begin{remark} In general, the Gromov width $\mathfrak{c}_{\on{Gr}}(X \times Y)$ is not even bounded by a multiple of $\mathfrak{c}_{\on{Gr}}(X)$. For instance, a result of Lalonde \cite{l1994} states that if $\Sigma$ is a closed surface, then
\[\mathfrak{c}_{\on{Gr}}(\Sigma \times \C) = \infty \quad\text{while} \quad \mathfrak{c}_{\on{Gr}}(\Sigma) = \text{area}(\Sigma)\]
From this, we deduce that the Gromov width of $X(a) = \Sigma \times S(a)$ where $S(a)$ is any symplectic surface of area $a$ diverges as $a \to \infty$.
\end{remark}

\subsubsection{Asymptotics} Many capacities, e.g. the ECH capacities and the Gutt--Hutchings capacities, come in natural families indexed by the integers. The asymptotic behavior of these capacities as the index goes to $\infty$ is of significant interest \cite{chr_asy_15,gh_sym_18,cs_sub_18,wor_tow_21} and is key to some of the dynamical applications \cite{i2015}. The RSFT capacities are naturally indexed by the codimension of the tangency constraint, and it is natural to ask how these capacities behave as the codimension diverges. 

\vspace{3pt}

In \S \ref{sec:algebraic_bounds}, we use an algebraic reformulation of $\lk_k$ and $\uk_k$, taking as input a polarized algebraic surface, to analyze their asymptotic behavior. In particular, we prove the following formula.

\begin{lemma}[Lemma \ref{lem:toric_asymtotics}] \label{lem:toric_asymtotics_intro} Let $\Omega$ be a strongly convex moment domain. Then
\[
\lim_{k\to\infty}\frac{\mfk{l}_k(\Omega)}{k}=\lim_{k\to\infty}\frac{\mfk{u}_k(\Omega)}{k}= \inf_{(w_1,w_2)\in\Z^2_{\geq0} \setminus 0}\frac{\|(w_1,w_2)\|_\Omega^*}{1+w_1+w_2}
\]
\end{lemma}

\noindent By combining Lemma \ref{lem:toric_asymtotics_intro} with Theorems \ref{thm:main_lowerbound} and \ref{thm:main_upperbound}, we deduce the following asymptotic formula.

\begin{thm*} \label{thm:intro_asymptotics} Let $\Omega \subset \R^2$ be a strongly convex moment domain and let $P_i$ be a sequence of lax tangency constraints with $\on{codim}(P) \to \infty$. Then
\begin{equation} \label{eqn:intro_asymptotics} \frac{\mathfrak{r}_{P_i}(X_\Omega)}{\on{codim}(P_i)} \to \inf_{(w_1,w_2)\in\Z^2_{\geq0}\setminus 0}\frac{\|(w_1,w_2)\|_\Omega^*}{2(1+w_1+w_2)} \qquad\text{as}\qquad i \to \infty \end{equation}
\end{thm*}

In the toric setting, the analogous limit of the Gutt--Hutchings capacities coincides with the Lagrangian torus capacity of Cieliebak--Mohnke \cite{cm2018}. Similarly, the analogous limit of the ECH capacities is the volume. Thus, we pose the following natural question.

\begin{question*} Does the righthand side of (\ref{eqn:intro_asymptotics}) coincide with a known symplectic capacity?
\end{question*}

\subsection*{Outline} This concludes {\bf \S 1}, the introduction. The rest of the paper is organized as follows.

\vspace{4pt}

In {\bf \S 2}, we cover preliminaries in tangency conditions (\S \ref{subsec:tangency_constraints}), rational SFT (\S \ref{subsec:RSFT}), Gromov--Witten theory (\S \ref{subsec:gromov_witten_invariants}) and RSFT capacities (\S \ref{subsec:axiomatic_capacities}). We establish several properties of the tangency Gromov--Witten invariants and RSFT capacities that are not explicitly proven in \cite{sie_hig_19,ms_cou_19}.

\vspace{4pt}

In {\bf \S 3}, we cover preliminaries in (toric) algebraic geometry. We review intersection theory (\S \ref{subsec:intersection_theory}), polarized varieties (\S \ref{subsec:polarized_varities}) and toric varieties (\S \ref{subsec:toric_varieties}). We also introduce cocharacter curves (\S \ref{subsec:cocharacter_curves}) and elucidate their relationship with the movable cone.

\vspace{4pt} 

In {\bf \S 4}, we formally introduce the bounds $\mathfrak{l}_k$ and $\mathfrak{u}_k$. We give four equivalent definitions: an algebro-geometric one (\S \ref{subsec:algebraic_formulation}), a toric (or fan) one (\S \ref{subsec:fan_formulation}), a lattice one (\S \ref{subsec:lattice_formulation}) and a polytope one (\S \ref{subsec:polytope_formulation}). We conclude by discussing the asymptotics of the bounds (\S \ref{subsec:asymptotics}).

\vspace{4pt}

In {\bf \S 5}, we prove Theorem \ref{thm:main_lowerbound} (\S \ref{subsec:proof_of_lower_bound}). As preparation, we discuss the Reeb dynamics of cosphere bundles (\S \ref{subsec:cosphere_bundles}) and free toric domains (\S \ref{subsec:free_domain}).

\vspace{4pt}

In {\bf \S 6}, we prove Theorems \ref{thm:main_upperbound}, \ref{thm:main_upperbound_closed} and \ref{thm:main_upperbound_stable}. The main tool is a construction of immersed symplectic spheres from cocharacter curves (\S \ref{subsec:torus_knots}-\ref{subsec:symplectic_spheres}).

\subsection*{Acknowledgements} We would like to thank Kyler Siegel and Dusa McDuff for several helpful discussions near the start of this project. JC was supported by the NSF Graduate Research Fellowship under Grant No.~1752814.

\section{Rational SFT capacities} \label{sec:RSFT_capacities}

In this section, we review of the construction of the rational SFT capacities of Siegel \cite{sie_hig_19,sie_com_19} and the Gromov--Witten invariants with tangency constraints of McDuff--Siegel \cite{ms_cou_19}. 

\begin{remark} We assume familiarity with basic concepts in symplectic and contact geometry, including Reeb flows, Conley-Zehnder indices, etc. 
\end{remark}

\subsection{Tangency constraints} \label{subsec:tangency_constraints} We begin with a brief discussion of tangency constraints, which play a significant part in both the symplectic field theory and the Gromov--Witten theory of this paper. 

\subsubsection{Basic definitions} We will treat this concept in a very formal way. 

\begin{definition} A \emph{tangency constraint} $P$ is a sequence of sequences of positive integers
\[
P = (P^1,\dots,P^m) \quad\text{where}\quad P^i = (P^i_1,\dots,P^i_{k_i}) \quad \text{and}\quad P^i_j \le P^i_{j+1}
\]
and an integer $\on{dim}(P) \ge 1$ called the \emph{dimension} of $P$. The \emph{number of points} $\# P$ is the number $m$.
\end{definition}

\begin{remark} We view a tangency constraint $P$ intuitively as follows. Consider a manifold $X$ of dimension $\on{dim}(P)$. Choose a distinct point and a germ of a codimenion $2$ sub-manifold
\[p_i \in X \qquad\text{and}\qquad D_i \subset X\]
A smooth immersion $u:\Sigma \to X$ satisfies the tangency condition $P$ if for the choice of $(p,D)$ if
\[u(s^i_j) = p_i \qquad\text{and}\qquad u \text{ is tangent to $D_i$ at order $P^i_j + 1$ at $s^i_j$}\]
for some set of points $s^i_j$ corresponding to each integer $P^i_j$. See Figure \ref{fig:tangency_condition} for a depiction. \end{remark}

\begin{figure}[h]
\centering
\caption{The tangency condition $P = ((0,1),(0,0),(2),(1,1))$ in a manifold $X$. The green arcs represent the pieces of a surface $\Sigma$ intersecting the divisors $D_i$.}
\includegraphics[width=.8\textwidth]{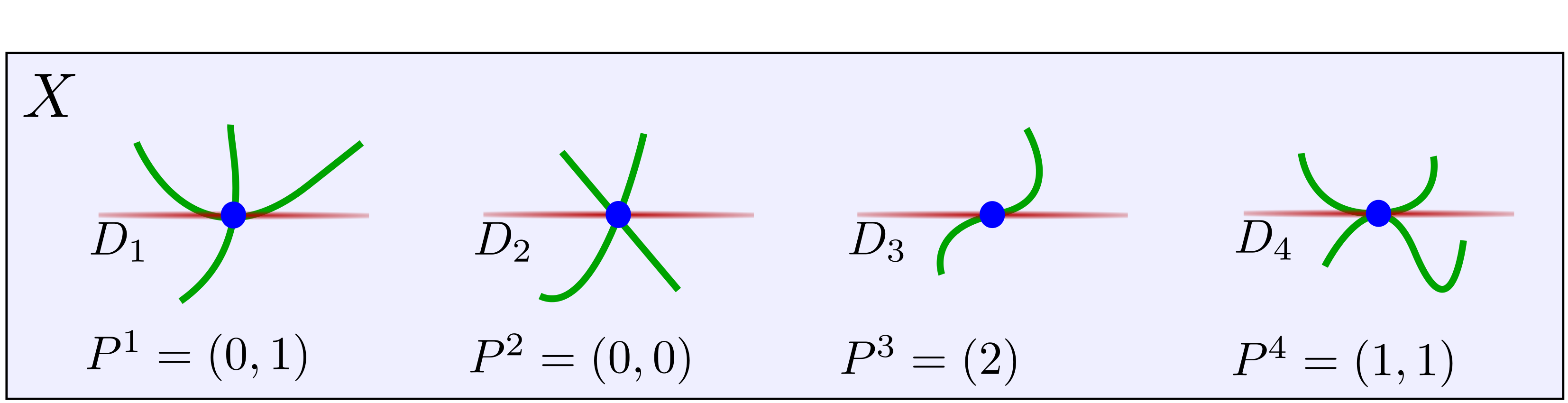}
\label{fig:tangency_condition}
\end{figure}

Every tangency constraint has an associated codimension, which measures the amount that the dimension of a moduli space of surfaces is cut down by when that constraint is imposed.

\begin{definition} The \emph{codimension} $\on{codim}(P)$ of a tangency constraint is defined by the formula
\begin{equation}
\on{codim}(P) = \sum_{i,j} \on{dim}(P) - 2 + 2P^i_j
\end{equation}
\end{definition}

\subsubsection{Operations} There are two useful monoidal operation on tangency constraints. The first may be viewed as adding together two sets of tangency constraints at the same set of points. 

\begin{definition} The \emph{catenation} $P \bullet Q$ of two tangency constraints $P$ and $Q$ with the same dimension and number of points is given by
\[P \bullet Q = (P^1 \bullet Q^1, \dots, P^m \bullet Q^m) \qquad \on{dim}(P \bullet Q) = \on{dim}(P)\]
where $P^i \bullet Q^i$ is the list acquired by joining the lists $P^i$ and $Q^i$, then ordering the result.\end{definition}

\noindent The second may be viewed as imposing the constraints $P$ and $Q$ on a surface $\Sigma$ simultaneously at a disjoint set of points for $P$ and $Q$.

\begin{definition} The \emph{union} $P \sqcup Q$ of two tangency constrainst $P$ and $Q$ with $\on{dim}(P) = \on{dim}(Q)$ is
\[P \sqcup Q = (P^1,\dots,P^m,Q^1,\dots,Q^n) \qquad\text{if}\qquad P = (P^1,\dots,P^m)\quad \text{and}\quad Q = (Q^1,\dots,Q^n)\]
The dimension $\on{dim}(P \sqcup Q)$ is just the dimension of the constraint $P$ (or equivalently of $Q$).
\end{definition}

\noindent Note that the number of points and codimension behave additively under catenation and union.

\begin{figure}[h]
\centering
\caption{An example of catenation and union of two $1$ point tangency constraints.}
\includegraphics[width=.8\textwidth]{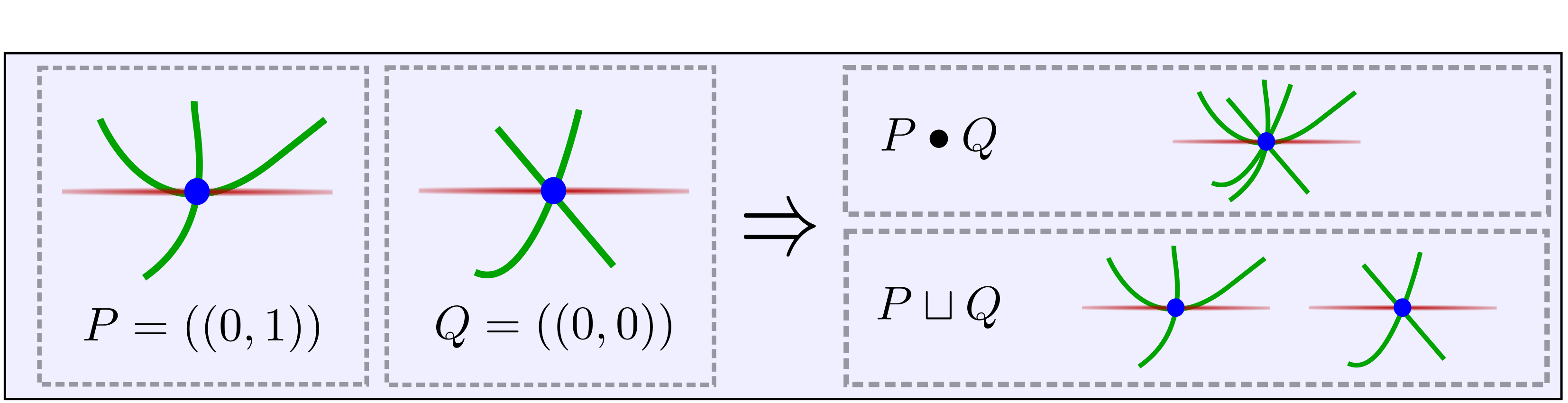}
\label{fig:tangency_catenation_union}
\end{figure}

It is useful to use catenation to form a graded algebra of tangency constraints.

\begin{definition} The \emph{tangency algebra} $\mathcal{T}^m_n$ is the graded $\Q$-module generated by the tangency constraints $P$ at $m$ points in dimension $n$, with
\[\text{product given by }\bullet \qquad\text{and}\qquad \text{grading given by } \on{codim}(-)\]\end{definition}

\noindent The union operation provides a natural, associative isomorphism of graded algebras
\[
T^a_n \otimes T^b_n \xrightarrow{\sim} T^{a + b}_n \qquad\text{with}\qquad P \otimes Q \mapsto P \sqcup Q
\]

\subsubsection{Laxness} Finally, we require the notion of a lax tangency constraint, which will serve as a helpful simplifying hypothesis throughout this paper.

\begin{definition} \label{def:laxness} A tangency condition $P$ is \emph{lax} if each sequence $P^i$ is length $1$.
\end{definition}

\begin{remark} The motivation for the terminology here is that any somewhere immersed surface $\Sigma \subset X$ in our hypothetical smooth manifold $X$ can be made to satisfy a lax tangency constraint by appropriatly choosing the tangency points $p$ and local divisors $D$. \end{remark}

\subsection{Rational symplectic field theory} \label{subsec:RSFT} We next briefly review of the construction and properties of rational SFT as presented by Siegel \cite{sie_hig_19}. This is primarily to establish notation and terminology.

\begin{remark}[Choices] The constructions of RSFT requires various standard choices of compatible complex structures and regularizing data for solving transversality issues (e.g. VFC data via Pardon's VFC framework \cite{p2015}). We ignore this data in the following discussion.
\end{remark}

\subsubsection{Basic formalism} The RSFT chain groups of a contact manifold $Y$ with contact form $\alpha$ are formulated as follows \cite{sie_hig_19}. 

\vspace{3pt}

First, we associate a generator to each good Reeb orbit $\gamma$ with a standard SFT grading and action filtration.
\begin{equation} x_\gamma \qquad \text{with}\qquad |x_\gamma| = -\on{CZ}(\gamma) + n - 3 \mod 2\quad\text{and}\quad\mathcal{A}(x_\gamma) = \int_\gamma \lambda \end{equation}
The \emph{contact algebra} $A(Y)$ is the graded-symmetric algebra freely generated by the elements $x_\Gamma$.
\begin{equation}
A(Y) = \on{Sym}_\bullet\big[x_\gamma \; : \; \gamma\text{ is a good orbit}\big]
\end{equation}
As a $\Q$-module, $A(Y)$ is freely generated by elements $x_{\Gamma} = x_{\gamma_1}\dots x_{\gamma_k}$ where $\Gamma = (\gamma_1,\dots,\gamma_k)$ is an unordered list of good orbits satisfying
\begin{equation} \label{eqn:orbit_set_condition} \text{$\gamma_i$ is not repeated if $|\gamma_i| = 1 \mod 2$}\end{equation}
The \emph{rational SFT bar complex} $BA(X)$ is the graded symmetric bialgebra freely generated by $A(Y)$ with a $+1$ grading shift. That is, we have
\begin{equation} 
BA(Y) = \on{Sym}_\bullet\big[ sx_\Gamma \; : \; \Gamma \text{ satsifies }(\ref{eqn:orbit_set_condition}) \big] \quad\text{where}\quad |sx_\Gamma| = |x_\Gamma| - 1
\end{equation}
As a $\Q$-module, $BA(Y)$ is freely generated by elements $x_{\bar{\Gamma}} = sx_{\Gamma_1} \ocirc \dots \ocirc sx_{\Gamma_k}$, where $\ocirc$ is the product in $BA(Y)$ and $\bar{\Gamma} = (\Gamma_1,\dots,\Gamma_k)$ is an unordered list such that
\begin{equation} \label{eqn:seq_of_orbit_sets}
\Gamma_i \text{ satisfies }(\ref{eqn:orbit_set_condition}) \qquad\text{and}\qquad \text{$\Gamma_i$ is not repeated if $|\Gamma_i| = 0 \mod 2$}
\end{equation}

As usual in SFT, the differentials and cobordism maps on RSFT bar complexes are defined using holomorphic curve counts in symplectic cobordisms. Given a symplectic cobordism $X:Y_+ \to Y_-$ and (sequences of) orbit sets $\bar{\Gamma}_\pm \subset Y_\pm$, we can consider finite energy holomorphic curves (or more generally, buildings) in the completion of $X$.
\[
u:\Sigma \to \hat{X} \qquad \text{with}\qquad u \to \Gamma_\pm \text{ at }\pm\infty
\]
The Fredholm index of a building $u$ only depends on $\bar{\Gamma}$, $\bar{\Xi}$, $\Sigma$ and the homology class $A$ of $u$ in the set $H_2(X,\bar{\Gamma} \cup \bar{\Xi})$ of relative classes with $\partial A = [\bar{\Gamma}_+] - [\bar{\Xi}_-]$. The formula is
\[
\on{ind}(A) = (n-3)\chi(\Sigma) + \on{CZ}_\tau(\bar{\Gamma}_+) - \on{CZ}_\tau(\bar{\Gamma}_-) + 2c_1(A,\tau)
\]
Here $\on{CZ}_\tau$ and $c_1(-,\tau)$ are the Conley-Zehnder index and 1st Chern class with respect to a trivialization of $\xi$ along $\bar{\Gamma}$ and $\bar{\Xi}$.

\vspace{3pt}

Schematially, the differential on the bar complex $BA(Y)$ is given by
\begin{equation}\partial:BA(Y) \to BA(Y) \qquad \text{with}\qquad \partial x_{\bar{\Gamma}} = \sum_{\bar{\Xi}} \; \# \bar{\mathcal{M}}(\bar{\Gamma},\bar{\Xi})/\R \cdot x_{\bar{\Xi}}\end{equation}
Here $\# \bar{\mathcal{M}}_Y(\bar{\Gamma},\bar{\Xi})/\R$ is a signed and weighted count of points in the Fredholm index $1$ components of a moduli space $\bar{\mathcal{M}}_Y(\bar{\Gamma},\bar{\Xi})/\R$ of certain disconnected, genus $0$ holomorphic buildings in the symplectization of $Y$ with positive ends on $\bar{\Gamma}$ and negative ends on $\bar{\Xi}$. Likewise, a symplectic cobordism $X:Y_+ \to Y_-$ induces a cobordism map of RSFT bar complexes 
\begin{equation}\Phi_X:BA(Y) \to BA(Z)
\qquad\text{with}\qquad
\Phi_X(x_{\bar{\Gamma}}) = \sum_{\bar{\Xi}} \; \#\bar{\mathcal{M}}_X(\bar{\Gamma};\bar{\Xi}) \cdot x_{\bar{\Xi}}\end{equation}
Here $\# \bar{\mathcal{M}}_X(\bar{\Gamma},\bar{\Xi})$ is a signed and weighted count of points in the Fredholm index $0$ components of a moduli space $\bar{\mathcal{M}}_X(\bar{\Gamma},\bar{\Xi})$ of certain disconnected, genus $0$ holomorphic buildings in the completion of $X$.

\begin{remark} The precise curves and combinatorial factors that are counted in the construction of $\partial$ and $\Phi$ is most easily explained using the formalism of $L_\infty$-algebra \cite[\S 3.4]{sie_hig_19}. The details of that formalism are unnecessary for this paper. \end{remark}

\begin{figure}[h]
\centering
\caption{A type of disconnected genus $0$ building that could be counted in the cobordism map in the $x_{\bar{\Xi}}$-coefficient of $\Phi_X(x_{\bar{\Gamma}})$.}
\includegraphics[width=.8\textwidth]{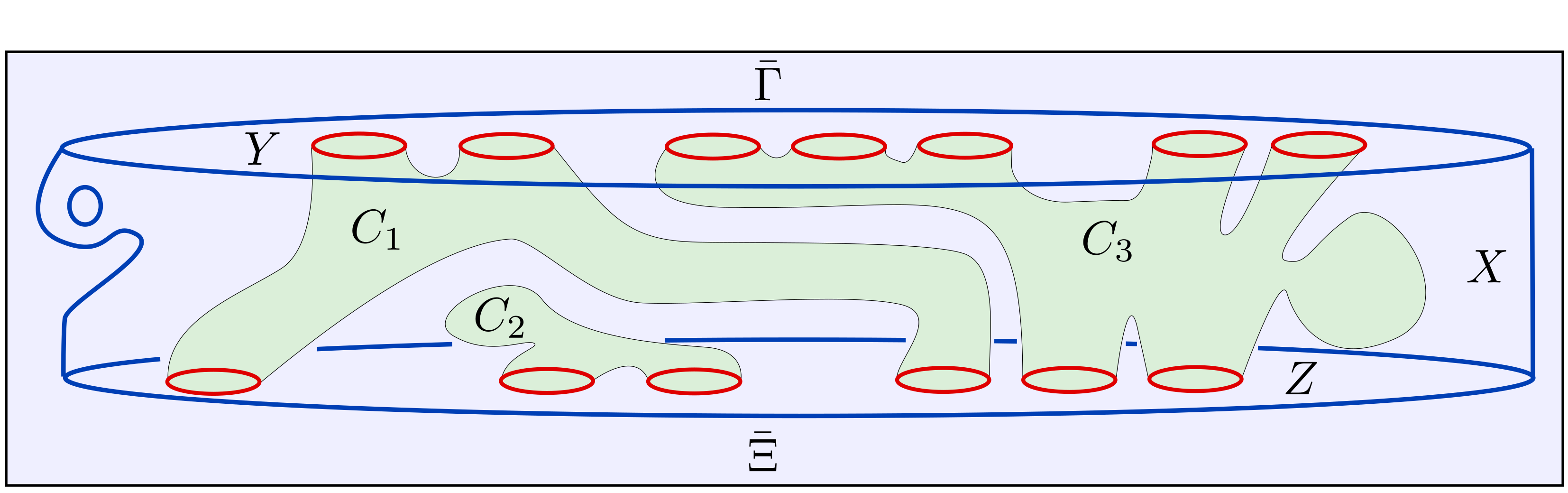}
\label{fig:holomorphic_curve_for_cobordism_map}
\end{figure}

These differentials and cobordism maps respect the action filtration. Furthermore, chain homotopies defined using parametrized curve counts can be used to show that the cobordism maps are well-defined and respect cobordism composition up to filtered chain homotopy.

\subsubsection{The tight sphere} 

The most basic example of a contact manifold is the sphere $S^{2n-1}$, and in this case the RSFT bar complex take a particularly simple form. 

\begin{lemma} \label{lem:RSFT_of_sphere} Let $S^{2n-1}$ denote the $(2n-1)$-sphere with the tight contact structure. Then there are canonical (up to homotopy) quasi-isomorphisms
\[
A(S^{2n-1}) \xrightarrow{\sim} \mathcal{T}_n \qquad BA(S^{2n-1}) \xrightarrow{\sim} B\mathcal{T}_n
\]
\end{lemma}

\subsubsection{Constrained Cobordism Maps} We now introduce an important generalization of the cobordism maps in RSFT involving the introduction of tangency constrains. 

\vspace{3pt}

First, fix a closed contact manifold $Y$ and let $S_1,\dots,S_m$ be a set of $m$ tight contact spheres of dimension $\on{dim}(Y)$. Note that there exists a degree $\on{codim}(S)$ chain map
\begin{equation} \label{eqn:projection_map_contact_algebra}
A(Y \cup S_1 \cup \dots \cup S_m) \to A(Y) \otimes \Q[\on{codim}(S)] = A(Y)[\on{codim}(S)]\end{equation}
This map is a straightforward extension of the map
\[
A(S_1 \cup \dots \cup S_m) \simeq \bigotimes_i A(S_i) \xrightarrow{\sim} T^{\otimes m}_n \simeq \mathcal{T}^m_n \xrightarrow{\Pi_P}  \Q[\on{codim}(P)]
\]
Here $\Pi_P$ maps element of $T^m_n$ to its $P$ coefficient. This map of contact dg-algebras extends naturally to a map of bar complexes.

\begin{definition}  The degree $\on{codim}(P)$ chain map
\begin{equation}\pi_P:BA(Z \cup S_1 \cup \dots \cup S_m) \to BA(Z)[\on{codim}(P)] \end{equation}
is the unique extension of (\ref{eqn:projection_map_contact_algebra}) to a coalgebra map on $BA(Y \cup S_1 \cup \dots \cup S_m)$. \end{definition}

Next, let $X:Y \to Z$ be a symplectic cobordism and consider the symplectic cobordism
\begin{equation} X(m):Y \to Z \cup S_1 \cup \dots \cup S_m \qquad\text{with}\qquad X(m) = X \setminus B_1 \cup \dots \cup B_m\end{equation}
where $B_i \subset X$ is a small ball and $S_i$ is the boundary $\partial B_i$.

\begin{definition} Let $P$ be a tangency constraint at $m$ points and let $X:Y \to Z$ be a symplectic cobordism. The \emph{$P$-constrained cobordism map}
\[\Phi_{X,P}:BA(Y) \to BA(Z)[\on{codim}(P)]\]
is the chain map of degree $\on{codim}(P)$ given by the composition
\[
BA(Y) \xrightarrow{\Phi_{X(m)}} BA(Z \cup S_1 \cup \dots \cup S_m) \xrightarrow{\pi_P} BA(Z)
\]

\end{definition}

\begin{definition} Let $P$ be a tangency constraint and $Y$ be a contact manifold. Then the \emph{U-map}
\[U_P:BA(Y) \to BA(Y)[\on{codim}(P)]\]
is the $P$-constrained cobordism map $\Phi_{X,P}$ where $X = [0,1] \times Y$ is the trivial cobordism.\end{definition}

\noindent Due to the functoriality of RSFT cobordism maps, we have the following composition laws for constrained cobordism maps (up to filtered homotopy).
\begin{equation} \label{eqn:constrained_cobordism_composition}
\Phi_{U \circ V, P \sqcup Q} \simeq \Phi_{U,P} \circ \Phi_{V,Q}
\end{equation}
Similar composition laws hold (up to homotopy) for the $U$-map as special cases of (\ref{eqn:constrained_cobordism_composition}).
\begin{equation}
\Phi_{X,P} \circ U_Q \simeq \Phi_{X,P \sqcup Q} \simeq U_Q \circ \Phi_{X,P} \qquad U_P \circ U_Q \simeq U_{P \sqcup Q}
\end{equation}
\vspace{3pt}

Observe that there is a map $\pi_*:H_2(X(m),\Gamma \cup P) \to H_2(X,\Gamma)$ for any orbit sets $\Gamma$ in $Y \cup Z$ and any tangency constraint $P$, viewed as an orbit set in $S_1 \cup \dots S_m$. It is useful to understand the relationship between the Fredholm indices of $A$ and $\pi_*A$ when using $\Phi_{X,P}$ in practice.

\begin{lemma} \label{lem:index_relation_punctured_cobordism} Let $A$ be a homology class in $X(m)$ with boundary $[\Gamma_+] - [\Gamma_-] - [P]$. Then
\begin{equation}
\on{ind}(A) = \on{ind}(\pi_*A) - \on{codim}(P)
\end{equation}
\end{lemma}

\begin{proof} Let $B_1 \cup \dots \cup B_m \subset X$ be the balls in $X$ such that $X(m) = X \setminus (B_1 \cup \dots B_m)$. By the additivity of the Fredholm index under cobordism composition, we have
\[
\on{ind}(\pi_*A) + \sum_{i,j} \on{ind}(D^i_j) = \on{ind}(A)
\]
Here $D^i_j \subset B_i$ is a disk in the ball $B_i$ that bounds the orbit $\gamma^i_j$ on $S_i = \partial B_i$ corresponding to the constraint index $P^i_j$ in $P$. The index of this disk is exactly
\[
\on{ind}(D^i_j) = (n-3) + \on{CZ}(\gamma^i_j) = (n-3) + (n - 1) + 2(P^i_j + 1) = 2n - 2 + 2P^i_j
\]  
The formula now follows from the definition of $\on{codim}(P)$. \end{proof}

\subsection{Gromov--Witten theory} \label{subsec:gromov_witten_invariants} We next discuss the Gromov--Witten invariants with tangency constraints introduced by McDuff--Siegel \cite{ms_cou_19}. Our discussion will be much less detailed than \S \ref{subsec:RSFT} since we will only need properties stated explicitly in \cite{ms_cou_19}.

\vspace{3pt}

Let $X$ be a closed symplectic manifold with a compatible almost complex structure and let $A \in H_2(X)$. Fix a set of $m = \# P$ points and consider the compactified moduli space
\[\bar{\mathcal{M}}_P(X,A)\]
of simple genus $0$ holomorphic curves in the homology class $A$ that satisfy the tangency constrain $P$ at the chosen points. The virtual (i.e. expected) dimension of the moduli space is
\begin{equation} \label{eqn:GW_index}
\on{ind}(A) - \on{codim}(P) = 2(n-3) + 2c_1(A) - \on{codim}(P)\end{equation}
In the dimension $0$ case, a usual signed count of points in the moduli space yields an invariant. 

\begin{thm} \cite[\S 2]{ms_cou_19} \label{thm:tangent_GW_invariants} Let $P$ be a tangency constrain, $X$ be a symplectic manifold and $A \in H_2(X)$ be a homology class with $\on{ind}(A) = \on{codim}(P)$. Then there is an associated Gromov--Witten invariant
\[
\on{GW}_P(X,A) \in \Q
\]
\end{thm}

\begin{remark} We warn the reader that Theorem \ref{thm:tangent_GW_invariants} is only handled in dimension $4$ by \cite[\S 2]{ms_cou_19}. The higher dimensional result requires advanced transversality methods and was not addressed. 
\end{remark}

We will require two special properties of the Gromov--Witten invariants $\on{GW}_P$. The first property is a consequence of Wendl's automatic transversality.

\begin{lemma}[Immersed Spheres] \label{lem:immered_spheres} \cite[Cor 2.3.9]{ms_cou_19} Let $P$ be an $m$-point tangency constraint and let $S$ be an immersed symplectic sphere in a symplectic $4$-manifold $X$ with positive self-intersections. Suppose
\[\text{$S$ satisfies $P$ at $m$ points} \qquad\text{and}\qquad \on{ind}(A) = \on{codim}(P)\]
Then the Gromov--Witten invariant $\on{GW}_P(X,A)$ is positive.
\end{lemma}

The second property is a certain stabilization property. Consider the tangency constraint $((k))$, which constrains a surface to pass through a local divisor at one point to order $k+1$.

\begin{lemma}[Stabilization] \label{lem:GW_stabilization} Let $X$ and $Y$ be closed, semi-positive symplectic manifolds and let $A \in H_2(X)$. Then
\[\on{GW}_{((k))}(X \times Y,A \otimes [\on{pt}]) = \on{GW}_{((k))}(X,A) \qquad\text{if}\qquad \on{ind}(A) = \on{codim}((k)) = (2n-2) + 2k\]
\end{lemma}

\begin{proof} Choose a point $p \in X$ and a local divisor $D$ through $p$. Let $P$ denote the $k$th order tangency constraint $((k))$ at $p$. Since $X$ is semi-positive, we know by \cite[Prop 2.2.2]{ms_cou_19} that we can choose a compatible almost complex structure $J$ such that the moduli spaces of genus $0$ $J$-curves
\[\mathcal{M}^s(X,A) \quad\text{and}\quad \mathcal{M}_P(X,A) \subset \mathcal{M}^s(X,A) \qquad\text{are transversely cut out}\]
Here $\mathcal{M}^s$ is the moduli space of simple (somwhere injective) curves. Note that the inclusion on the right only holds when $\on{ind}(A) = \on{codim}(P)$, so that the dimension of $\mathcal{M}_P(X,A)$ is $0$. 

\vspace{3pt}

Now choose a compatible complex structure $I$ on $Y$ and consider the corresponding moduli spaces in the product $X \times Y$ with respect to the product almost complex structure $J \times I$.
\[
\mathcal{M}(X \times Y, A \otimes [\on{pt}]) \quad\text{and}\quad \mathcal{M}_P(X \times Y,A \otimes [\on{pt}]) 
\]
There is an obvious bijection of moduli spaces given by
\begin{equation} \label{eqn:stab_lemma_1}
\mathcal{M}^s(X,A) \times Y \xrightarrow{\sim} \mathcal{M}^s(X \times Y, A \otimes [\on{pt}]) \qquad\text{with}\quad ([u],y) \mapsto [u,y]\end{equation}
For any $[v] \in \mathcal{M}^s(X \times Y,A \otimes [\on{pt}])$ pullback of the tangent bunde of $X \times Y$ by $v$ splits as follows.
\[
v^*T(X \times Y) = u^*TX \oplus p^*TY \simeq u^*TX \oplus \C^m
\]
This corresponds to a splitting of the linearized operator of $[v]$ into a direct sum of the linearized operator for $u$ and the del-bar operator on the trivial bundle $u^*TY \simeq \C^m$. The latter is transverse (surjective) and (real) Fredholm index $2m$. Therefore (\ref{eqn:stab_lemma_1}) is a diffeomorphism and intertwines the Fredholm orientation lines.

\vspace{3pt}

Choosing a point $q \in Y$ and a stabilized local divisor $D \times B \subset X \times Y$ where $B$ is a ball around $q$. Then the map (\ref{eqn:stab_lemma_1}) restricts to a diffeomorphism of $P$-constrained moduli spaces
\begin{equation} \label{eqn:stab_lemma_2}
\mathcal{M}_P(X,A) \xrightarrow{\sim} \mathcal{M}_P(X \times Y, A \otimes [\on{pt}]) \quad\text{with}\quad [u] \mapsto [u,q]
\end{equation}
One may check that the virtual dimension is preserved. Indeed, if $\on{dim}(X) = 2n$ and $\on{dim}(Y) = 2m$, then
\[
\on{vdim} \mathcal{M}_P(X \times Y, A \otimes [\on{pt}]) = 2(m+n - 3) + c_1(A) - (2(m+n) - 2) - 2k\]
\[= 2(n - 3) + c_1(A) - (2n-2) - 2k = \on{vdim} \mathcal{M}_P(X, A)\]
Furthermore, the map intertwines Fredholm orientation lines and thus yields an equality of signed point counts. In other words, $\on{GW}_P(X,A) = \on{GW}_P(X \times Y, A \otimes [\on{pt}])$.
\end{proof}

\begin{remark} Note that every closed symplectic manifold of dimension less than or equal to $6$ is semi-positive. \end{remark}

\subsection{Axiomatic RSFT capacities} \label{subsec:axiomatic_capacities} Now that we have provided the reader with a basic outline of rational symplectic field theory and Gromov--Witten theory, we are now prepared to introduce the RSFT capacities. We will treat these capacities axiomatically, via the following theorem.

\begin{thm} \label{thm:main_rP_axioms} \cite{sie_hig_19} For each tangency constraint $P$, there is a corresponding rational SFT capacity
\[
\mathfrak{r}_P(X) \qquad \text{for each Liouville domain}\qquad (X,\lambda)
\]
Moreover, the capacities $\mathfrak{r}_P$ satisfy the following axioms.

\vspace{3pt}
\begin{itemize}
	\item[(a)] (Monotonicity) Let $X \to Y$ be a symplectic embedding of Liouville domains. Then
	\[\mathfrak{r}_P(X) \le \mathfrak{r}_P(Y)\]
	\item[(b)] (Sub-Additivity) Let $P = Q \sqcup R$ be the union tangency constraint of $Q$ and $R$. Then
	\[\mathfrak{r}_Q(X) \le \mathfrak{r}_P(X)\]
	\item[(c)] (Reeb Orbits) If $\lambda|_{\partial X}$ is non-degenerate, then there is a list of Reeb orbits $\Gamma$ on $\partial X$ and a (possibly disconnected) genus $0$ immersed surface $\Sigma \subset X$ bounding $\Gamma$ such that
	\[
	 \mathcal{A}(\Gamma) = \sum_{\gamma \in \Gamma} \int_\gamma \lambda \le \mathfrak{r}_P(X) \qquad\text{and}\qquad \on{ind}(\Sigma) = \on{codim}(P)
	\] 
	\item[(d)] (Gromov--Witten) Let $X \to Y$ be an embedding of $X$ into a closed symplectic manifold $(Y,\omega)$ and let $A \in H_2(Y)$ be a homology class. Assume that $H_1(\partial X) = H_2(\partial X) = H_2(X) = 0$. Then
	\[
	\on{GW}_P(Y,A) \neq 0 \qquad\text{implies}\qquad \mathfrak{r}_P(X) \le [\omega] \cdot A
	\]
\end{itemize}
\end{thm}

\begin{proof} The capacities $\mathfrak{r}_P$ are constructed in \cite[\S 6.1]{sie_hig_19}. In our notation, the definition is
\begin{equation}\label{eqn:capacity_def} \mathfrak{r}_P(X) = \on{min}\big\{\mathcal{A}(x) \; : \; x \in BA(X) \text{ with }\partial x = 0 \text{ and } \Phi_{X,P}(x) \neq 0\big\}\end{equation}
The monotonicity and Gromov--Witten properties are proven explicitly in \cite[\S 6.2.1]{sie_hig_19} and \cite[\S 6.2.3]{sie_hig_19}, respectively. We prove (b) and (c), which are unstated in \cite{sie_hig_19}.

\vspace{5pt}

\emph{(b) - Sub-Additivity.} Property (b) is an easy application of the existence of the $U$ maps. Namely, let $x \in BA(X)$ be closed and let $y = U_R(x)$. Then
\begin{equation} \label{eqn:subadditivity_proof}
\Phi_{X,Q}(y) = \Phi_{X,Q} \circ U_R(x) = \Phi_{X,P}(x) \neq 0 \qquad\quad \mathcal{A}(y) \le \mathcal{A}(x)
\end{equation}
The sub-additivity under union of tangency constraint is now a consequence of (\ref{eqn:capacity_def}) and (\ref{eqn:subadditivity_proof}).

\vspace{5pt}

\emph{(c) - Reeb Orbits.} Let $x$ be a closed element of the bar complex $BA(X)$ that satisfies
\[\Phi_{X,P}(x) \neq 0 \qquad\quad \mathcal{A}(x) = \mathfrak{r}_P(X)\]
Expand $x$ as a finite sum of generators $x_{\bar{\Gamma}}$ of $BA(X)$, where $\bar{\Gamma} = (\Gamma_1,\dots,\Gamma_m)$ is a sequence of orbit sets satisfying (\ref{eqn:seq_of_orbit_sets}) with action bounded above by $\mathcal{A}(x)$.
\[x = \sum_{\bar{\Gamma}} c_{\bar{\Gamma}} \cdot x_{\bar{\Gamma}} \qquad\text{with}\qquad c_{\bar{\Gamma}} = 0 \text{ if }\mathcal{A}(\bar{\Gamma}) \le \mathcal{A}(x)\]
By the linearity of the constrained cobordism map $\Phi_{X,P}$, we must have that
\begin{equation} \label{eqn:constrained_map_on_x_Gamma} \Phi_{X,P}(x_{\bar{\Gamma}}) = \pi_P \circ \Phi_{X(m)}(x_{\bar{\Gamma}}) \neq 0 \qquad \text{for some} \qquad \bar{\Gamma} \text{ with }c_{\bar{\Gamma}} \neq 0\end{equation}
The map $\pi_P \circ \Phi_{X(m)}(x_{\bar{\Gamma}})$ counts disconnected genus $0$ holomorphic buildings of index $0$ with negative ends on an orbit set identified with $P$ on the union of spheres $S_1 \cup \dots \cup S_m$. Such a building must $u$ exist by (\ref{eqn:constrained_map_on_x_Gamma}). Choose a smooth genus $0$ immersion $\phi:\Sigma(m) \to X(m)$ homotopic to (the full gluing of) such a building along with a capping $\phi:\Sigma \to X$ in $X$. Let
\[A = [\Sigma(m)] \in H_2(X(m),\bar{\Gamma} \cup P) \qquad\text{and}\qquad \pi_*A = [\Sigma] \in H_2(X,\bar{\Gamma})\]
Then by the index relation in Lemma \ref{lem:index_relation_punctured_cobordism}, we see that
\[
0 = \on{ind}(A) = \on{ind}(\pi_*A) - \on{codim}(P)
\]
Thus, the surface $\Sigma$ and the orbit set $\Gamma = \cup_i \Gamma_i$ satisfy the properties in (c). \end{proof}

\section{Algebraic and toric geometry} \label{sec:algebraic_and_toric_geometry} 

In this section, we review the aspects of (toric) algebraic geometry that we will need in this paper. In particular, we discuss the class of cocharacter curves that we will use to bound the $\mfk{r}_P$.

\subsection{Interection theory} \label{subsec:intersection_theory} We will consider several cones associated to a projective variety $Y$, which live inside the homology of $Y$ when $Y$ is sufficiently nonsingular. By an \emph{algebraic cycle} of dimension $k$ on $Y$ we mean a formal $\Z$-linear combination of irreducible algebraic subvarieties of dimension $k$.

\vspace{3pt}

There is an important equivalence relation on the set of algebraic cycles of dimension $k$ on $Y$ called numerical equivalence: two cycles $Z=\sum c_iZ_i$ and $W=\sum d_iW_i$ are numerically equivalent if the equality
$$\sum c_i\on{deg}(Z_i\cap S)=\sum d_i\on{deg}(W_i\cap S)$$
holds for any irreducible subvariety $S\subset Y$ of codimension $k$. We denote by $N_k(Y)_\Z$ the abelian group of algebraic cycles of codimension $k$ considered up to numerical equivalence. We also set $N_k(Y)_\bK$ to denote the tensor product $N_\bullet(Y) \otimes_\Z \bK$ for $\bK\in\{\Q,\R\}$.

\vspace{3pt}

If $Y$ has dimension $n$ there is an intersection pairing
$$N_k(Y)_\bK\otimes N_{n-k}(Y)_\bK\to\R$$
which provides a identification $N_k(Y)_\Z\simeq N_{n-k}(Y)_\Z^\vee$ when $Y$ is smooth.

\vspace{3pt}

There are several natural cones within $N_k(Y)_\R$ associated to $Y$. The \emph{effective cone} is the cone generated by the classes of irreducible subvarieties.
\[\on{Eff}_k(Y) := \on{Cone}\big([Z] \; : \; Z \subset Y \text{ irreducible subvariety}\big) \subset N_k(Y)\]
Similarly, the \emph{big, pseudo-effective} and \emph{movable} cones of $Y$ are the interior, closure and dual cones to the effective cone, respectively.
\[\on{Big}_k(Y) := \on{int}(\on{Eff}_k(Y)) \qquad \on{\ol{Eff}}_k(Y) \qquad \on{Mov}_k(Y) := \on{Eff}_{n-k}(Y)^\vee\]
In the special case of curves (i.e.~$k=1$) we adopt the following standard notation:
\[\on{NE}(Y) := \on{Eff}_1(Y) \qquad \on{\ol{NE}}(Y) := \on{\ol{Eff}}_1(Y) \qquad\on{Nef}(Y) := \on{Eff}_1(Y)^\vee \subset N_{n-1}(Y)\]
The closure $\on{\ol{NE}}(Y)$ is called the \emph{Mori cone} and the dual $\on{Nef}(Y)$ is called the \emph{nef cone}. We add a subscript $\Z$ to denote the intersection of any of these cones with $N_1(Y)_\Z \subset N_1(Y)_\R$. 

\subsection{Polarized varieties} \label{subsec:polarized_varities} We will be particularly interested in projective varieties equipped with a distinguished divisor that plays a similar role to that of a symplectic form for a symplectic manifold.

\begin{definition} 
A \emph{pseudo-polarized variety} $(Y,A)$ is a pair consisting of a projective variety $Y$ and a big and nef $\Q$-Cartier $\R$-divisor $A$ on $Y$. We say that $(Y,A)$ is a \emph{polarized variety} if $A$ is ample.
\end{definition} 

A divisor $Z$ is $\Q$\emph{-Cartier} if it is an integer combination of irreducible divisors $Z_i$ such that $d \cdot Z_i$ is the vanishing locus of a section of a line bundle on $Y$ for some integer $d$. $Y$ is $\Q$\emph{-factorial} if every divisor is $\Q$-Cartier.

\subsection{Toric varieties} \label{subsec:toric_varieties} A \emph{toric variety} $Y$ is a projective variety admitting an action of $(\C^\times)^n$ that is free and faithful on a Zariski open set $T \subset Y$ called the \emph{big torus}. We take most of the following from \cite{cls_tor_11} and the abundant references therein.

\vspace{3pt}

The \emph{character lattice} $M$ and \emph{cocharacter lattice} $N$ are the free abelian groups
\begin{equation}
M := \on{Hom}(T,\C^\times) \qquad N := \on{Hom}(\C^\times,T)
\end{equation}
In particular, $N$ is the lattice of $1$-parameter subgroups of $T \simeq (\C^\times)^n$. We have a natural identification $N \simeq M^\vee$ and a choice of basis for either yields isomorphisms $N \simeq \Z^n\simeq M$.
\vspace{3pt}

Every (normal) toric variety $Y$ is determined by a fan $\Sigma$ in $N$: a collection of finitely generated cones $\sigma \subset N_\R:=N\otimes_\Z\R \simeq \R^n$ containing the origin and closed under intersection and taking faces. Without comment we will assume that every cone in $\Sigma$ contains a point in $N$ as is standard in toric geometry. We write $Y_\Sigma$ for the toric variety corresponding to a fan $\Sigma$. We let
\[\Sigma(r) \subset \Sigma \qquad\text{and}\qquad \Sigma(u,r) \subset \Sigma(r) \text{ for }u \in N\]
denote, respectively, the cones in $\Sigma$ of dimension $r$ and the cones in $\Sigma(r)$ contained in the minimal cone $\sigma \in \Sigma$ with $u \in \sigma$. Finally, for any $\rho \in \Sigma(1)$ we let $v_\rho$ denote the unique primitive vector in $\rho\cap N$, which we call the \emph{generator} of the ray.

\vspace{3pt}

Divisor and curve classes admit a simple description in terms of the fan. Namely, there is a short exact sequence for curves
\begin{equation}
0 \to N_1(Y) \to \Z^{\Sigma(1)} \xrightarrow{\pi} N \to 0 \quad\text{where}\quad \pi(\rho) = v_\rho
\end{equation}
There is also a dual exact sequence for divisors.
\begin{equation} \label{eqn:toric_divisor_seq}
0 \to M \to \Z^{\Sigma(1)} \to N_{n-1}(Y) \to 0
\end{equation}
which expresses that each ray $\rho \in \Sigma(1)$ corresponds to a natural $(\C^\times)^n$-invariant divisor $D_\rho \subset Y$, and that these divisors span $N_{n-1}(Y)_\Z$.

\vspace{3pt}

Define the \emph{support function} $\ph_D\colon N_\R \to [0,\infty)$ for a divisor $D=\sum_{\rho\in\Sigma(1)} a_\rho D_\rho$ on $Y$ by 
$$\ph_D(v_\rho)=a_\rho$$
and extending linearly on each cone of $\Sigma$. This procedure identifies $\on{Nef}(Y)$ with the cone of `strongly convex' functions that are piecewise linear on the fan $\Sigma$. We can also associate a (possibly empty) polytope $P(D) \subset M_\R$ to $D$ as
\[P(D) := \big\{u \in M_\R \; : \; \langle u,v_\rho\rangle \le \ph_D(v)\big\}\]
The global sections of the sheaf $\mO(D)$ correspond to lattice points in $P(D)$. We say that a polytope $\Omega$ is \emph{supported} by a fan $\Sigma$ if it arises via this construction. This is equivalent to saying that the fan $\Sigma$ refines the fan generated by the outward normal vectors of $\Omega$. Polytopes and support functions are equivalent objects for nef divisors (i.e. they determine each other) and so can be used interchangeably. Lastly, $D$ is big if and only if its polytope $P(D)$ is full-dimensional in $M_\R$. 

\vspace{3pt}

We say that a polytope $\Omega\subset M_\R$ is \emph{rational-sloped} if each facet (codimension $1$ face) has a normal in $N$. One can reverse the process described above to produce from a full-dimensional, rational-sloped $\Omega$ a pseudo-polarized toric variety $(Y_\Sigma,A_\Omega)$, where $Y_\Sigma$ is the toric variety corresponding to the fan spanned by the outward normal vectors of $\Omega$ and $A_\Omega$ is the big and nef divisor with polytope $P(A_\Omega)=\Omega$. As discussed (dually) in \cite[\S4.3]{wor_ech_19}, the support function for $A_\Omega$ is the \emph{$\Omega$-norm} $||\cdot||_\Omega^*$ defined by
$$||v||_\Omega^*:=\sup_{u\in\Omega}\langle u,v\rangle$$

\vspace{3pt}

We say that a fan $\Sigma$ in $\R^n$ is \emph{smooth} if the generators of the rays in each cone of $\Sigma$ form part of a $\Z$-basis of $\Z^n$. This corresponds to the toric variety $Y_\Sigma$ being smooth. One can `refine' any fan $\Sigma$ -- further subdividing the cones in $\Sigma$ -- to produce a smooth fan $\wt{\Sigma}$. This corresponds to a torus-equivariant resolution of singularities $\pi\colon Y_{\wt{\Sigma}}\to Y_\Sigma$.

\begin{example} \label{ex:p112} Consider the fans shown in Figure \ref{fig:p112}. The fan on the left describes the weighted projective space $\pr(1,1,2)$. This has a singularity locally isomorphic to $\C^2/\pm1$ captured by the fact that the ray generators for the top left cone are $(-1,0)$ and $(1,2)$, which do not form a $\Z$-basis of $\Z^2$. This singular toric variety can be equivariantly resolved by a single blowup in a torus-fixed point, corresponding to adding the extra ray spanned by $(0,1)$ shown on the right.

\begin{figure}[h]
\centering
\caption{The fan for $\pr(1,1,2)$ and its resolution}
\includegraphics[width=.6\textwidth]{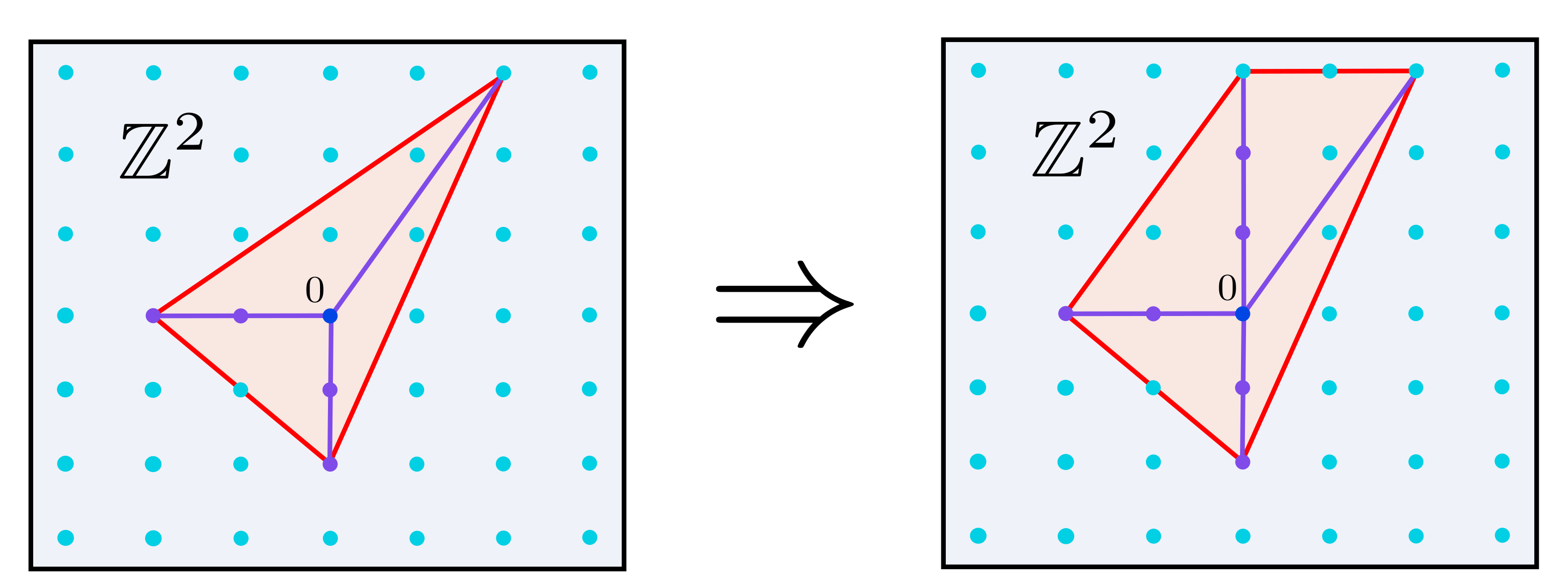}
\label{fig:p112}
\end{figure}
\end{example}

We conclude this subsection by defining the class of toric varieties we will focus on in our main results.

\begin{definition} Fix an isomorphism $N\simeq \Z^n$. A fan $\Sigma$ on $\R^n$ is \emph{strongly convex} if $\Sigma(1)$ consists entirely of rays $\rho$ of the form
\[\on{Cone}(-e_i) \qquad\text{or}\qquad \rho \subset \R_{\geq0}^n\]
We let $\Sigma^+(1) \subset \Sigma(1)$ denote the set of rays satisfying the second condition. A polytope $\Omega$ or the associated pseudo-polarized toric variety $(Y_\Sigma,A_\Omega)$ is \emph{strongly convex} if $\Omega$ is supported by a strongly convex fan.
\end{definition}

Observe that both of the fans in Example \ref{ex:p112} are strongly convex. It is straightforward to check that one can resolve a singular strongly convex toric variety to a smooth strongly convex toric variety as illustrated in that example.

\vspace{3pt}

Notice that a fan is actually only strongly convex relative to the isomorphism $N\simeq \Z^n$. This however will be immaterial for our purposes since an isomorphism $M\simeq\Z^n$ is fixed by a choice of moment map.

\subsection{Cocharacter curves} \label{subsec:cocharacter_curves} We will discuss how a cocharacter $u \in N$ determines a rational curve in $Y$. The $2$-homology classes of these curves have a rich Gromov--Witten theory and will play an important role later in the paper.

\begin{definition} The \emph{cocharacter curve} $C_u$ associated to a cocharacter $u:\C^\times \to T$ of $Y$ is the curve acquired by taking the scheme-theoretic closure of the image $u(\C^\times)$.

\vspace{3pt}

By abuse of notation, we let $C_\rho := C_{v_\rho}$ for any ray $\rho \in \Sigma(1)$. 
\end{definition}

\begin{remark} The only reference we know of that alludes to these curves is \cite[Prop 1.6]{rei_dec_83} with their appearance made more explicit in the exposition of this result in \cite[\S1]{wis_tor_02}. \end{remark}

\begin{example} When $Y=\pr^n$ the curves $C_v$ are the monomial curves; that is, they are given by the closure of a map of the form $t\mapsto(t^{a_1},\dots,t^{a_n})$ into the usual affine chart $\C^n \subset \pr^n$. For instance, the map $t\mapsto(t^2,t^4)$ gives a $2$-fold cover of a degree $2$ curve.
\end{example}

Expressing the curve class of a cocharacter is simple. We adopt the following definition.

\begin{definition} The \emph{cocharacter relation} $R_u \in \Z^{\Sigma(1)}$ is the vector defined uniquely by the relations
\begin{equation}
u = \sum_{\rho \in \Sigma(u,1)} R_{u,\rho} \cdot v_\rho \qquad\qquad -u = \sum_{\rho \in \Sigma(-u,1)} R_{u,\rho} \cdot v_\rho\end{equation}
and the property that $R_{u,\rho} = 0$ for any ray $\rho \in \Sigma(1)$ that is not in $\Sigma(u,1)\cup\Sigma(-u,1)$.
\end{definition}

If $u=v_\rho$ we set $R_u=:R_\rho$. By examining intersection numbers between $C_u$ and the torus-invariant divisors $D_\rho$ one can easily show the following.

\begin{lemma} \label{lem:cocharacter_class} The curve class $[C_u] \in N_1(Y)$ of the cocharacter curve $C_u$ is identified with the vector $R_u$ under the isomorphism $N_1(Y) \simeq \on{ker}(\Z^{\Sigma(1)} \to N)$.
\end{lemma}

It is also straightforward to compute intersection numbers of a cocharacter curve with any divisor by using the support function of the divisor.

\begin{lemma}[Intersection] \label{lem:cocharacter_intersection} Let $C_u$ be a cocharacter curve and let $D$ be a Cartier divisor on $Y$. Then
\begin{equation}
C_u\cdot D=\ph_D(u)+\ph_D(-u)
\end{equation}
\end{lemma}

\begin{proof} Recall we may write the class of $D$ in terms of the torus-invariant divisors as follows.
\[D=\sum_{\rho \in \Sigma(1)} \ph_D(v_\rho)D_\rho\]
The curve $C_u$ only meets the toric divisors corresponding to rays in $\Sigma(u,1)$ and $\Sigma(-u,1)$ by the orbit-cone correspondence \cite[\S3.2]{cls_tor_11}. We now compute directly that
\[C_u\cdot D = \sum_{\rho\in\Sigma(1)} R_{u,\rho} \cdot \varphi_D(v_\rho) = \sum_{\rho \in \Sigma(u,1)} R_{u,\rho} \cdot \varphi_D(v_\rho) + \sum_{\rho \in \Sigma(-u,1)} R_{u,\rho} \cdot \varphi_D(v_\rho)\]
We then apply the linearity of $\ph_D$ on cones of $\Sigma$ to find that
\[C_u \cdot D = \ph_D\big(\sum_{\rho\in\sigma_u(1)} a_\rho v_\rho\big)+\ph_D\big(\sum_{\rho\in\sigma_{-u}(1)}b_\rho v_\rho\big) =\ph_D(u)+\ph_D(-u)\]
where $\sigma_v$ denotes the smallest cone in $\Sigma$ containing $v$.
\end{proof}

As a consequence of Lemma \ref{lem:cocharacter_intersection} we obtain the following special cases.

\begin{cor}[Area] \label{cor:cocharacter_area} Let $A$ be a big and nef divisor on $Y$ with polytope $P(A)=\Omega$. Then
\[A_\Omega\cdot C_u  = ||u||_\Omega^*+||-u||_\Omega^*\]
\end{cor}

\begin{cor}[Index] \label{cor:cocharacter_index} Let $K_Y$ be the canonical divisor of $Y$. Then
\[-K_Y \cdot C_u = \sum_{\rho\in\Sigma(1)} R_{u,\rho}\]
\end{cor}

We will focus the main results of this paper on the strongly convex setting, coming primarily from the following result.

\begin{lemma}[Movable Curves] \label{lem:strongly_convex_movable} Let $Y$ be a smooth toric variety given by a strongly convex fan $\Sigma$. Then
\[\on{Mov}_1(Y)_\Z = \left\{\sum_{\rho\in\Sigma^+(1)}b_\rho[C_\rho] : b_\rho\in\Z_{\geq0}\right\}\]
\end{lemma}

\begin{proof} Let $C$ be a movable curve class corresponding to a vector $(b_\rho)_{\rho\in\Sigma(1)}\in\on{ker}(\Z^{\Sigma(1)} \to M)$. Let $\rho_1,\dots,\rho_n$ be the rays generated by $-e_1,\dots,-e_n$. We have
\[
0 = \sum_{i=1}^n b_{\rho_i} \cdot (-e_i) + \sum_{\rho \in \Sigma^+(1)} b_\rho \cdot v_\rho \qquad\text{and}\qquad b_{\rho} \ge 0 \text{ for each }\rho
\]
Since $v_\rho$ has non-negative entries for every $\rho \in \Sigma^+(1)$ we must be able to decompose this relation as
\[
0 = \sum_{\rho \in \Sigma^+(1)} b_\rho \cdot (v_\rho - \sum_{i=1}^n v_{\rho,i} \cdot e_i)=\sum_{\rho \in \Sigma^+(1)} b_\rho \cdot R_\rho
\]
and so $C=\sum_{\rho\in\Sigma^+(1)}b_\rho C_\rho$.
\end{proof}

We will also need a description of the curves in a toric surface with self-intersection zero. 

\begin{lemma}[Fiber Classes] \label{lem:fiber_classes} Let $Y$ be a smooth toric surface given by a strongly convex fan $\Sigma$ and let $F \in \on{Mov}_1(Y)$ be a movable curve class with $F \cdot F = 0$. Then
\[F = m \cdot [C_\rho] \quad\text{where}\quad \rho = \on{Cone}(e_i) \text{ and }m \ge 0 \]
\end{lemma}

\begin{proof} By Lemma \ref{lem:strongly_convex_movable}, the cocharacter curves $C_\rho$ generate the movable cone. It thus suffices to show that $C_\rho^2=0$ if and only if $v_\rho=e_1$ or $v_\rho=e_2$. By direct computation, if $\rho = \on{Cone}(e_1)$ or $\rho = \on{Cone}(e_2)$ are in $\Sigma$, then the corresponding cocharacter curve $C_\rho$ has zero self-intersection. In particular, $C_\rho$ is the divisor $D_\sigma$ corresponding to the rays spanned by $\sigma = -e_2$ (in the first case) or $\sigma = -e_1$ (in the second case). 

\vspace{3pt}

On the other hand, suppose that $v_\rho$ has positive entries, and so is not a multiple of $e_i$ for $i = 1,2$. Then note that the torus-fixed point $p_0$ corresponding to $\on{Cone}(-e_1,-e_2)$ is in $C_\rho$ but that $C_\rho$ is not torus-invariant. Perturbing $C_\rho$ using the torus-action preserves $p_0$, and so $C_\rho^2>0$.
\end{proof}

\section{Algebraic bounds} \label{sec:algebraic_bounds} In this section, we introduce lower bounds $\mathfrak{u}_k$ and upper bounds $\mathfrak{l}_k$ for the rational SFT capacities. We also establish several important properties, e.g. asymptotic behavior with $k$. 

\vspace{3pt}

We present four different formulations of these bounds: an \emph{algebraic} formulation (\S \ref{subsec:algebraic_formulation}), a \emph{fan} formulation (\S \ref{subsec:fan_formulation}), a \emph{lattice} formulation (\S \ref{subsec:lattice_formulation}) and a \emph{polytope} formulation (\S \ref{subsec:polytope_formulation}). Each formulation is useful in different circumstances and all agree on their common domains of definition. 

\subsection{Algebraic formulation} \label{subsec:algebraic_formulation} The first formulation of our bounds is applicable to any $\Q$-factorial projective surface $Y$ equipped with a big and nef $\R$-divisor $A$.

\vspace{3pt}

Consider the movable cone $\on{Mov}_1(Y)$, and let $\on{Mov}_1(Y)^+ \subset \on{Mov}_1(Y)$ denote the subset
\begin{equation}
\on{Mov}_1(Y)^+ :=\{C\in\on{Mov}_1(Y):C^2 >0\}\cup\{F:\text{$F$ is irreducible and $F\cdot F=0$}\}
\end{equation}
Observe that the curves $F$ in the righthand set are automatically movable.

\begin{definition}[Algebraic bounds] Let $(Y,A)$ be a (possibly singular) $\Q$-factorial pseudo-polarized surface. We adopt the definition
\begin{equation}
\lk_k(Y,A):=\inf\{\pi^*A \cdot C \; : \; C \in\on{Mov}_1(\wt{Y})_\Z \text{ and }-K_Y\cdot C\geq k+1\} \end{equation}
\begin{equation}\uk_k(Y,A):=\inf\{\pi^*A\cdot C\; : \; C \in\on{Mov}_1(\wt{Y})^+_\Z \text{ and }-K_Y\cdot C\geq k+1\}
\end{equation}
Here $\pi\colon\wt{Y}\to Y$ is any resolution of singularities. If $Y$ is smooth one can just take $\pi$ to be the identity map.
\end{definition}

The following result shows that the definition for singular surfaces is independent of resolution; compare to \cite[Prop 3.4]{wor_alg_20}.

\begin{prop}[Blowup Formula] \label{prop:blowup} Let $(Y,A)$ be a pseudo-polarized smooth surface, and let $\pi\colon\wt{Y}\to Y$ be a blowup with exceptional divisor $E$. Fix $\eps\geq0$ such that $\wt{A}_\eps=\pi^*A-\eps E$ is big and nef. Then
\[\mfk{l}_k(Y,A)\geq\mfk{l}_k(\wt{Y},\wt{A}_\eps)\quad \text{and}\quad \mfk{l}_k(Y,A)=\mfk{l}_k(\wt{Y},\pi^*A)\]
\[\mfk{u}_k(Y,A)\geq\mfk{u}_k(\wt{Y},\wt{A}_\eps)\quad \text{and}\quad \mfk{u}_k(Y,A)=\mfk{u}_k(\wt{Y},\pi^*A)\]
\end{prop}

\begin{proof} We provide the proof for $\lk_k$ and then explain the modifications for $\uk_k$. Suppose $C$ realises $\mfk{l}_k(Y,A)$. Then $\pi^*C$ is movable on $\wt{Y}$ and $-K_{\wt{Y}}\cdot C=-K_Y\cdot C\geq k+1$ giving
\[\mfk{l}_k(\wt{Y},\wt{A}_\eps)\leq\wt{A}_\eps\cdot\pi^*C=A\cdot C=\mfk{l}_k(Y,A)\]
Now suppose $\wt{C}=\pi^*C-mE$ realises $\mfk{l}_k(\wt{Y},\pi^*A)$. Note $m\geq0$ to ensure $E\cdot\wt{C}\geq0$. We see
\[k+1\leq-K_{\wt{Y}}\cdot(\pi^*C-mE)=-K_Y\cdot C-m\leq-K_Y\cdot C\]
\[\text{and hence} \qquad \mfk{l}_k(Y,A)\leq A\cdot C=\pi^*A\cdot\wt{C}=\mfk{l}_k(\wt{Y},\pi^*A)\]
The converse inequality is obtained by setting $\eps=0$ in the first argument of the proof, and this proves the result for $\lk_k$. 

\vspace{3pt}

We proceed similarly for $\uk_k(Y,A)$. The only additional complexity is the exclusion of the fiber classes. If $C^2>0$ then clearly $\pi^*C\in\on{Mov}(\wt{Y})_\Z^+$. The remaining case is when $C=[F]$ for an irreducible $0$-curve $F\subset Y$. We can assume that the centre of the blowup is not on $F$ and so $\pi^*F$ is also an irreducible curve of self-intersection zero and hence in $\on{Mov}(\wt{Y})_\Z^+$, giving the result.
\end{proof}

This algebraic formulation is useful for importing techniques from birational geometry, such as in studying the asymptotics of RSFT capacities as we will see in \S\ref{subsec:asymptotics}.

\subsection{Fan formulation} \label{subsec:fan_formulation} The second formulation of our bounds applies to any pseudo-polarized toric surface $(Y_\Sigma,A_\Omega)$, and is given in terms of the corresponding (outer normal) fan $\Sigma$.

\vspace{3pt}

We start by defining subsets $L(\Xi,k)$ and $U(\Xi,k)$ of the lattice $\Z^{\Xi(1)}$ for any fan $\Xi$. Namely, let $L(\Xi,k) \subset \Z^{\Xi(1)}$ be the set of vectors $a \in \Z^{\Xi(1)}$ that satisfy
\begin{equation} \label{eqn:U_k_conditions}
\sum_{\rho \in\Xi(1)} a_\rho v_\rho = 0 \qquad \sum_\rho a_\rho \ge k + 1 \quad\text{and}\quad a_\rho \ge 0 \qquad \text{for each }\rho \in \Xi(1) \end{equation}
Let $F(\Xi) \subset \Z^{\Xi(1)}$ be the set consisting of vectors $a$ such that there is a $j \in \{1,2\}$ with
\[a_\rho=m>0\text{ if $v_\rho=\pm e_j$} \qquad  \text{and}\qquad a_\rho=0\text{ else}
\]
Here $e_1$ and $e_2$ are the standard basis vectors of $\Z^2$. Finally, we define $U(\Xi,k) \subset \Z^{\Xi(1)}$ by
\begin{equation}
U(\Xi,k) = L(\Xi,k) \setminus F(\Xi)
\end{equation}

\begin{lemma}[Fan formula] \label{lem:fan_formula} Let $\Omega$ be a convex rational-sloped polytope supported by the fan $\Sigma$. Then
\[\lk_k(Y_\Sigma,A_\Omega) = \inf_{(a_\rho) \in L(\widetilde{\Sigma},k)}\Big\{ \sum_\rho a_\rho ||v_\rho||_\Omega^*\Big\} \qquad\quad \uk_k(Y_\Sigma,A_\Omega) = \inf_{(a_\rho) \in U(\widetilde{\Sigma},k)}\Big\{ \sum_\rho a_\rho ||v_\rho||_\Omega^*\Big\} \]
Here $\widetilde{\Sigma}$ is the refinement of $\Sigma$ corresponding to a toric resolution $\widetilde{Y}_\Sigma$ of $Y_\Sigma$.
\end{lemma}

\begin{proof} Recall that a curve $C$ in $\on{Mov}_1(\wt{Y}_\Sigma)$ is identified with an element $a \in\Z^{\wt{\Sigma}(1)}$ in the kernel of the natural map $\Z^{\wt{\Sigma}(1)} \to N$ with $a_\rho \ge 0$ for each $\rho$. Moreover, by Corollary \ref{cor:cocharacter_area} and Corollary \ref{cor:cocharacter_index} the area and index are given by
\[
A_\Omega \cdot [C] = \sum_\rho a_\rho \cdot \|v_\rho\|_\Omega^* \qquad\text{and}\qquad -K_\Sigma \cdot C = \sum_\rho a_\rho 
\]
Thus the desired formula for $\mfk{l}_k$ follows from the definition of $L(\widetilde{\Sigma},k)$. By Lemma \ref{lem:fiber_classes} the reducible self-intersection zero curves in $\on{Mov}_1(Y_\Sigma)$ are precisely represented by the subset $F(\wt{\Sigma}) \cap \on{Mov}_1(\wt{Y}_\Sigma) \subset \Z^{\wt{\Sigma}(1)}$. Therefore
\[
\mathfrak{u}_k(Y_\Sigma,A_\Omega)=\inf_{(a_\rho)\in U(\widetilde{\Sigma},k)}\Big\{\sum_\rho a_\rho \cdot \|v_\rho\|_\Omega^* \Big\}
\]
This concludes the proof. \end{proof}

\subsection{Lattice formulation} \label{subsec:lattice_formulation} The third lattice-based formulation of our bounds is defined for any convex toric domain in $\C^2$. This version was introduced already in the introduction (see \S \ref{subsec:main_results}).

\vspace{3pt}

We begin (as in \S \ref{subsec:fan_formulation}) by defining distinguished sets of sequences in $\Z^2$, denoted $L(k)$ and $U(k)$. Namely, let $L(k)$ denote the set of sequences
\begin{equation}
\bar{v} = v_1,\dots,v_m \in \Z^2 \setminus 0 \qquad\text{with}\qquad m \ge k+1 \text{ and }\sum_i v_i = 0
\end{equation}
Likewise, let $U(k) \subset L(k)$ denote the subset of sequences $\bar{v} = v_1,\dots,v_m$ that satisfy
\[
m = 2 \text{ and } v_1 = -v_2 = e_j \text{ for }j = 1,2 \qquad\text{or}\qquad \text{ for each $j = 1,2$, there is an $i$ with }v_i \not\in \text{span}(e_j)
\]
The lattice version of our bounding quantities can now be defined by the following optimizations.

\begin{definition} \label{def:lattice_def_lk_uk} Let $\Omega \subset \R^n$ be a convex domain containing $0$. We define
\begin{equation}\lk_k(\Omega) := \inf_{\bar{v} \in L(k)}\Big\{ \sum_{i=1}^{m} ||v_i||_\Omega^* \Big\} \qquad \uk_k(\Omega) := \inf_{\bar{v} \in U(k)}\Big\{ \sum_{i=1}^{m} ||v_i||_\Omega^* \Big\}\end{equation}
\end{definition}

The lattice versions of $\mathfrak{l}_k$ and $\mathfrak{u}_k$ have several useful features similar to capacities.

\begin{lemma}[Properties] \label{lem:lk_uk_properties} The quantities $\mathfrak{l}_k(-)$ and $\mathfrak{u}_k(-)$ satisfy the following properties.
\vspace{3pt}
\begin{itemize}
\item[(a)] (Monotonicity) If $\Omega \subset \Delta$ is an inclusion of convex domains then
\[\mathfrak{l}_k(\Omega) \le \mathfrak{l}_k(\Delta) \qquad\text{and}\qquad \mathfrak{u}_k(\Omega) \le \mathfrak{u}_k(\Delta)\]
\item[(b)] (Scaling) If $c \cdot \Omega$ is the scaling of $\Omega$ by some constant $c > 0$ then
\[\mathfrak{l}_k(c \cdot \Omega) = c \cdot \mathfrak{l}_k(\Delta) \qquad\text{and}\qquad \mathfrak{u}_k(c \cdot \Omega) = c \cdot \mathfrak{u}_k(\Omega)\]
\item[(c)] (Continuity) $\mathfrak{l}_k(\Omega)$ and $\mathfrak{u}_k(\Omega)$ are continuous in the Hausdorff metric on convex sets.
\end{itemize}
\end{lemma}

\begin{proof} If $\Omega \subset \Delta$, then $\| w \|_\Omega^* \le \| w \|_\Delta^*$ for any vector $w \in \R^2$, and this implies (a). Similarly, $\|w\|^*_{c \cdot \Omega} = c \cdot \|w\|^*_\Omega$ implies (b). Finally, (c) follows from (a) and (b) and a scaling argument.
\end{proof}

These lattice formulas agree with the algebraic and fan formulations introduced above, from which it is apparent why we needed to resolve singularities.

\begin{prop}[Lattice formula] Let $\Omega \subset \R^2$ be a convex polytope. Then
\[\lk_k(Y_\Sigma,A_\Omega) = \lk_k(\Omega) \qquad\qquad \mathfrak{u}_k(Y_\Sigma,A_\Omega) = \mathfrak{u}_k(\Omega)\]
\end{prop}

\begin{proof} We prove inqualities in either direction between the fan formula in Lemma \ref{lem:fan_formula} and the lattice formula in this proposition.

\vspace{3pt}

\emph{Lattice $\le$ Fan.} First we show that the lattice formulas bound the fan formulas from below.
\[\mathfrak{l}_k(\Omega) \le \mathfrak{l}_k(Y_\Sigma,A_\Omega) \qquad \mathfrak{u}_k(\Omega) \le \mathfrak{u}_k(Y_\Sigma,A_\Omega)\]
Indeed, let $a \in L(k) \subset \Z^{\Sigma(1)}$ and denote $m$ be the sum of the entries $a_\rho$. Partition $\{1,\dots,m\}$ arbitrarily into sets $I_\rho$ for each $\rho \in \Sigma(1)$. Let $v_1,\dots,v_m$ be the sequence in $\Z^2$ with $v_i = v_\rho$ for each $i \in I_\rho$. Then we have
\[
v_i \neq 0\text{ for each }i \qquad\text{and}\qquad \sum_\rho a_\rho \cdot \|v_\rho\|^*_\Omega = \sum_i \|v_i\|_\Omega^*
\]
The infimum defining $\mathfrak{l}_k(\Omega)$ in Definition \ref{def:lattice_def_lk_uk} thus implies that $\mathfrak{l}_k(\Omega) \le \mathfrak{l}_k(Y_\Sigma,A_\Omega)$. If $a \in U(k) \subset L(k)$, then there are two cases for the sequence $v_i$ constructed above. 
\begin{equation} \tag{1} m = 2 \text{ and }v_1 = -v_2 = e_j \qquad \text{ for }j = 1,2\end{equation}
\begin{equation} \tag{2} \text{ for each $j = 1,2$, there is an $i$ with }v_i \neq \pm e_j
\end{equation}
Case (1) corresponds to the irreducible, self-intersection $0$ case. Case (2) corresponds to the positive self-intersection case. These are precisely the cases considered in the infimum defining $\mathfrak{u}_k(\Omega)$ in Definition \ref{def:lattice_def_lk_uk}, and so $\mathfrak{u}_k(\Omega) \le \mathfrak{u}_k(Y_\Sigma,A_\Omega)$

\vspace{3pt}

\emph{Fan $\le$ Lattice.} Now we show that the fan formulas bound the lattice formulas from below.
\[\mathfrak{l}_k(Y_\Sigma,A_\Omega) \le \mathfrak{l}_k(\Omega) \qquad \mathfrak{u}_k(Y_\Sigma,A_\Omega) \le \mathfrak{u}_k(\Omega)\]
Let $v_1,\dots,v_m$ be a sequence of non-zero vectors in $\Z^2$ for $m \ge k+1$. Note that each vector $v_i$ in some cone $\sigma_i$ in the resolved fan $\wt{\Sigma}$ associated to $\Omega$. Thus, we can define
\begin{equation}
(a_\rho) \in \Z^{\wt{\Sigma}(1)} \quad\text{given by} \quad a_\rho := \sum_i a_{i,\rho} \quad\text{where} \quad v_i = \sum_{\rho \in \sigma_i(1)} a_{i,\rho} \cdot v_{\rho}
\end{equation}
Note that the coefficients $a_{i,\rho}$ are unique and integral since $\wt{\Sigma}$ is smooth. We now claim that $(a_\rho)$ is an element of $L(\Sigma,k) \subset \Z^{\wt{\Sigma}(1)}$. First, note that $a_\rho \ge 0$ since $a_{i,\rho} \ge 0$ for each $i$ and $\rho$. Second, note that $a_{i,\rho} \ge 1$ for each $i$ and some $\rho \in \sigma_i$, since $v_i$ is non-zero. Third, note that
\[\sum_\rho a_\rho \cdot v_\rho = \sum_\rho \big(\sum_i a_{i,\rho} \cdot v_\rho \big) = \sum_i \big(\sum_\rho a_{i,\rho} \cdot v_\rho\big) = \sum_i v_i = 0\]
This verifies the conditions in (\ref{eqn:U_k_conditions}). Finally, the $\Omega$-norm is linear on each cone in $\Sigma$, and so
\[\sum_\rho a_\rho \cdot \|v_\rho\|_\Omega^* = \sum_i \sum_{\rho \in \sigma_i} a_{i,\rho} \cdot \|v_\rho\|_\Omega^* = \sum_i \|\sum_{\rho \in \sigma_i} a_{i,\rho} \cdot v_\rho\|_\Omega^* = \sum_i \|v_i\|_\Omega^*\]
It now follows from Lemma \ref{lem:fan_formula} and Definition \ref{def:lattice_def_lk_uk} that $\mathfrak{l}_k(Y_\Sigma,A_\Omega) \le \mathfrak{l}_k(\Omega)$.

\vspace{3pt}

Now suppose that the sequence $v_1,\dots,v_m$ is admissible for calculating $\mathfrak{u}_k(\Omega)$. There are two cases. In case (1), $m = 2$ and $v_1 = -v_2 = e_j$ for some fixed $j = 1,2$. Therefore
\[a \in \Z^{\Sigma(1)} \quad\text{represents the curve class}\quad [C_\rho] \quad \text{with}\quad \rho = e_j\]
In particular, $(a_\rho)$ represents an irreducible curve of self-intersection $0$ by Lemma \ref{lem:fiber_classes} and is in $U(\Sigma,k)$. In case (2), for each $j = 1,2$ there is an index $i$ such that $v_i \neq e_i$. In particular,
\[a \in \Z^{\Sigma(1)} \quad\text{does not represent}\quad k \cdot [C_\rho] \quad\text{with}\quad \rho = e_j \text{ for }j = 1,2\text{ and }k \ge 2\]
It follows in this case also that $(a_\rho) \in U(\Sigma,k)$. We can now conclude from Lemma \ref{lem:fan_formula} and Definition \ref{def:lattice_def_lk_uk} that $\mathfrak{u}_k(Y_\Sigma,A_\Omega) \le \mathfrak{u}_k(\Omega)$. \end{proof}

We conclude this section by providing a few simple situations where $\mathfrak{l}_k$ and $\mathfrak{u}_k$ agree. First, the two quantities coincide when $k = 1$.

\begin{lemma} \label{lem:first_bound} Let $\Omega$ be any strongly convex moment domain. Then
\[\mathfrak{u}_1(\Omega) = \mathfrak{l}_1(\Omega) = \on{max}(a(\Omega),b(\Omega))\]
\end{lemma}

\begin{proof} Fix a vector $v \in \Z^2 \setminus 0$. Then we have $\|v\|^*_\Omega = 0$ if $v \cdot e_1 \le 0$ and $v \cdot e_2 \le 0$. Similarly, 
\[\|v\|^*_\Omega \ge (v \cdot e_1) a(\Omega) \text{ if }v \cdot e_1 > 0 \qquad\text{and}\qquad \|v\|^*_\Omega > (v \cdot e_2) b(\Omega) \text{ if }v \cdot e_2 > 0\]
With this in mind, fix a sequence of non-zero vectors $\bar{v} = (v_1,\dots,v_m)$ in $\Z^2$ with $v_1 + \cdots + v_m = 0$ and $m \ge 2$. We can assume without loss of generality that $v_1 \cdot e_1 > 0$ or $v_2 \cdot e_2 > 0$, so that
\[
\sum_i \|v_i\|_\Omega^* \ge  \text{min}(a(\Omega),b(\Omega))
\]
This minimum is achieved by one of the sequences $(e_1,-e_1)$ and $(e_2,-e_2)$. These sequences are admissible for $\mathfrak{u}_1$ and $\mathfrak{l}_1$. This implies the result.
\end{proof} 

Second, the two quantities coincide whenever the widths $a(\Omega)$ and $b(\Omega)$ of $\Omega$ (see \S \ref{subsec:intro_toric_domains}) are the same.

\begin{lemma} \label{lem:lk_uk_equal_when_a_is_b} Let $\Omega$ be a strongly convex moment domain with $a(\Omega) = b(\Omega)$. Then
\[\mathfrak{l}_k(\Omega) = \mathfrak{u}_k(\Omega)\]
\end{lemma}

\begin{proof} Fix a sequence $\bar{v} = v_1,\dots,v_m$ in $L(k)\setminus U(k)$. It suffices to find a sequence in $U(k)$ with the same $\Omega$-norm. By the definition of $U(k)$, we have $v_i = \pm e_j$ for some fixed $j = 1$ or $j = 2$. Without loss of generality, we assume that
\[
v_i = (-1)^i \cdot e_1
\]
Now let $\bar{w} \in U(k)$ be the sequence with
\[
w_i = (-1)^i e_1 \text{\; \; for\; \; }i = 1,2 \qquad\text{and}\qquad w_i = (-1)^i e_2 \text{\;\;for\;\; }3 \le i \le m
\]
Then since $\|e_1\|_\Omega^* = a(\Omega) = b(\Omega) = \|e_2\|_\Omega^*$, we have
\[
\sum_i \|w_i\|_\Omega^* = \sum_i \|v_i\|_\Omega^*
\]
This concludes the proof.
\end{proof}

\subsection{Polytope formulation} \label{subsec:polytope_formulation} Our last formulation is given in terms of lattice polytopes, and is valid for any pseudo-polarized toric variety $(Y_\Sigma,A_\Omega)$ in any dimension.

\begin{remark} This version plays no role in the rest of the paper. However, it is somewhat closer in spirit to other lattice formulas for the ECH capacities from \cite{ccfhr_sym_14,cri_sym_19,wor_ech_19}, and can be extended to make sense for higher dimensional toric varieties. Thus we have chosen to include it. \end{remark}

Any movable curve $C$ in the toric variety $Y_\Sigma$ has an associated polytope $P(C)$. Indeed, to a curve corresponding to a relation
\[a \in \Z^{\Sigma(1)} \qquad\text{with}\qquad \sum a_\rho v_\rho=0\]
The polytope $P(C)$ has a facet of (relative) volume $m_\rho$ and with normal $v_\rho$ for each $\rho\in\Sigma(1)$. One may verify that
\[-K_\Sigma \cdot C = \laff(\partial P(C)) \qquad A_\Omega \cdot C = \ell_\Omega(\partial P(C))\]
Here $\laff(\partial P(C))$ denotes the affine area of $P(C)$ and $\ell_\Omega(\partial P(C))$ denotes the $\Omega$-area of $P(C)$. In complex dimension $2$, $P(C)$ has an edge for each ray $\rho \in \Sigma(1)$ of length $m_\rho$ and orthogonal to $v_\rho$. In this case, $P(C)$ is equivalent to the polytope associated to $C$ as a nef divisor. 

\begin{lemma} Let $(Y_\Sigma,A_\Omega)$ be a pseudo-polarised toric surface. Then,
$$\mfk{l}_k(Y_\Sigma,A_\Omega)=\inf_P\{\ell_\Omega(\partial P):\laff(\partial P)\geq k+1\}$$
where $P$ ranges over all (possibly degenerate) convex lattice polygons. Similarly, for $k\geq 2$
$$\mfk{u}_k(Y_\Sigma,A_\Omega)=\inf_P\{\ell_\Omega(\partial P):\laff(\partial P)\geq k+1\}$$
where $P$ now ranges over all non-degenerate (i.e. $2$ dimensional) convex lattice polygons.
\end{lemma}

\begin{proof} By resolving it suffices to assume that $Y_\Sigma$ is smooth. We will consider the case of $\mfk{u}_k$; $\mfk{l}_k$ is studied very similarly. Note that the polytope $P$ is full-dimensional if and only if the corresponding curve class $C$ is big.  Thus, we must verify that one can range over all convex lattice polygons $P$ as opposed to only those whose facet normals define rays in the fan for $Y_\Sigma$. 

\vspace{3pt}

This follows from the blowup-invariance of $\mfk{l}_k$ and $\mfk{u}_k$ from Proposition \ref{prop:blowup}. Indeed, suppose a polygon $P_0$ is preferable to all polygons with all slopes in common with $\Omega$, but that $\partial P_0$ contains at least one facet with a normal vector that does not define a ray in the fan for $Y$. There exists a series of blowups $\pi\colon\wt{Y}\to Y$ such that the normals to all facets of $P_0$ are contained in the fan for $\wt{Y}$. Thus $P_0$ defines a big and movable curve $C_0$ on $\wt{Y}$. However, Proposition \ref{prop:blowup} implies that there exists a big and movable curve $C$ on $Y$ such that $\pi^*C$ is preferable to $C_0$. By construction, the facet normals of $\partial P(\pi^*C)$ are among the ray generators of the fan for $Y$ and so we are done.
\end{proof}

\subsection{Asymptotics} \label{subsec:asymptotics} We conclude this section by studying the asymptotic behavior of $\lk_k$ and $\uk_k$ in the limit as $k \to \infty$. 

\begin{prop}[Asymptotics] \label{prop:lk_uk_asymptotics} Let $(Y,A)$ be a pseudo-polarized smooth surface with effective anti-canonical divisor $-K_Y$. Then
\begin{equation} \label{eqn:lk_uk_asymptotics}
\lim_{k\to\infty}\frac{\mfk{l}_k(Y,A)}{k}=\lim_{k\to\infty}\frac{\mfk{u}_k(Y,A)}{k}=\inf_{C\in\on{Mov}_1(Y)_\Z}\frac{A\cdot C}{-K_Y\cdot C}
\end{equation}
\end{prop}

\begin{proof} First, we show that the $\lk_k$ limit agrees with the formula on the right of (\ref{eqn:lk_uk_asymptotics}). Assume without loss of generality $Y$ is smooth. Fix an $\eps > 0$ and choose a movable class $C_\eps$ that satisfies
\[\frac{A\cdot C_\eps}{-K_Y\cdot C_\eps} - \eps \leq \frac{A\cdot C}{-K_Y\cdot C} \quad\text{for any}\quad C \in \on{Mov}_1(Y)_\Z\]
Let $k_\eps=-K_Y\cdot C_\eps$. Then
\[-K_Y\cdot\lceil\tfrac{k+1}{k_*}\rceil C_\eps\geq k+1 \quad\text{so that}\quad\mfk{l}_k(Y,A)\leq A\cdot\lceil\tfrac{k+1}{k_*}\rceil C_\eps
\]
Hence
\[
\frac{\mathfrak{l}_k(Y,A)}{k} \leq \frac{A \cdot \lceil\tfrac{k+1}{k_\eps}\rceil C_\eps}{k} \leq \frac{k+1}{k}\cdot\frac{A \cdot C_\eps}{-K_Y \cdot C_\eps}+\frac{A\cdot C_\eps}{k}
\]
As $k\to\infty$ we find
\begin{equation} \label{eqn:asymptotics_ub}
\lim_{k\to\infty}\frac{\mfk{l}_k(Y,A)}{k}\leq\frac{A \cdot C_\eps}{-K_Y \cdot C_\eps}\leq\inf_{C\in\on{Mov}_1(Y)_\Z}\frac{A\cdot C}{-K_Y\cdot C}+\eps
\end{equation}
On the other hand, consider the analog of $\lk_k$ optimizing over movable $\R$-curves.
\[\mfk{l}_k^\R(Y,A):=\inf_{C\in\on{Mov}(Y)_\R}\{A\cdot C:-K_Y\cdot C\geq k+1\}\]
Choose a movable $\R$-curve $C$ with $-K_Y\cdot C\geq k+1$ and $A\cdot C\leq\mfk{l}_k^\R(Y,A)+\delta$. Then
\[\mfk{l}_k^\R(Y,A)\geq A\cdot C-\delta\geq(-K_Y\cdot C)\left(\frac{A\cdot C_\eps}{-K_Y \cdot C_\eps}-\eps\right)-\delta\geq(k+1)\left(\frac{A\cdot C_\eps}{-K_Y \cdot C_\eps}-\eps\right)-\delta\]
Since $\mfk{l}_k^\R(Y,A)\leq\mfk{l}_k(Y,A)$, the above equation implies that 
\begin{equation}\label{eqn:asymptotics_lb} \lim_{k\to\infty}\frac{\mfk{l}_k(Y,A)}{k}\geq\lim_{k\to\infty}\frac{(k+1)\left(\frac{A\cdot C_\eps}{-K_Y \cdot C_\eps}-\eps\right)-\delta}{k}=\frac{A\cdot C_\eps}{-K_Y \cdot C_\epsilon}-\eps \ge \inf_{C\in\on{Mov}(Y)_\Z}\frac{A\cdot C}{-K_Y\cdot C}-\eps\end{equation}
By taking $\eps\to0$ in (\ref{eqn:asymptotics_ub}) and (\ref{eqn:asymptotics_lb}), we acquire the desired formula for $\mathfrak{l}_k$. 

\vspace{3pt}

Second, we show that $\lk_k$ and $\uk_k$ limits coincide. Let $C_k$ be a sequence of moveable curves with $A \cdot C_k = \lk_k(Y,A)$. Let $B$ be an arbitrary big and moveable curve, so that $B_k := B + C_k$ is also big and moveable. In particular
\[-K_Y \cdot B_k \ge -K_Y \cdot C_k \ge k + 1\]
Therefore, $\uk_k(Y,A) \le A \cdot B_k$. It follows that
\[
\lim_{k \to \infty} \frac{\uk(Y,A)}{k} \le \lim_{k \to \infty} \frac{A \cdot B + A \cdot C_k}{k} = \lim_{k \to \infty} \frac{\lk_k(Y,A)}{k}
\]
The reverse inequality is obvious. Thus $\lk_k(Y,A)/k$ and $\uk_l(Y,A)/k$ have the same limit in $[0,\infty]$.
\end{proof}

Every toric surface has effective anticanonical divisor, and thus we have the following corollary.

\begin{lemma} \label{lem:toric_asymtotics_polytope} Let $\Omega\subseteq\R^2$ be a strongly convex rational-sloped polytope
. Then
\begin{equation} \label{eqn:lattice_asymp}
\lim_{k\to\infty}\frac{\mfk{l}_k(Y_\Sigma,A_\Omega)}{k}=\lim_{k\to\infty}\frac{\mfk{u}_k(Y_\Sigma,A_\Omega)}{k}= \underset{\rho \in \Sigma^+}{\on{min}} \frac{\|v_\rho\|^*_\Omega}{1 + v_{\rho,1} + v_{\rho,2}}
\end{equation}
\end{lemma}

\begin{proof} By Proposition \ref{prop:lk_uk_asymptotics}, the limits on the lefthand side exist and are given by
\begin{equation} \label{eqn:asymp_res}
\inf_{C\in\on{Mov}_1(Y_{\wt{\Sigma}})_\Z}\frac{\pi^*A_\Omega\cdot C}{-K_{\wt{\Sigma}}\cdot C}
\end{equation}
Here we have equivariantly resolved singularities via $\pi\colon Y_{\wt{\Sigma}} \to Y_\Sigma$ (if necessary). By Lemma \ref{lem:strongly_convex_movable}, for any $C\in\on{Mov}_1(Y_{\wt{\Sigma}})_\Z$ we can write
\[
\frac{\pi^*A_\Omega\cdot C}{-K_{\wt{\Sigma}}\cdot C} = \frac{\sum_\rho k_\rho (\pi^*A_\Omega \cdot C_\rho)}{\sum_\rho k_\rho (K_{\wt{\Sigma}} \cdot C_\rho)} \quad\text{where}\quad C = \sum_{\rho \in \wt{\Sigma}^+} k_\rho \cdot C_\rho 
\]
For any positive sequences $a_1,\dots,a_n$ and $b_1,\dots,b_n$ of real numbers, we have
\begin{equation} \label{eqn:elementary_sum_fact_1}
\sum_i a_i = \sum_i \frac{a_i}{b_i} \cdot b_i \ge \on{min}_i\big(\frac{a_i}{b_i}\big) \cdot \sum_i b_i\end{equation}
Applying (\ref{eqn:elementary_sum_fact_1}) and Corollaries \ref{cor:cocharacter_area}-\ref{cor:cocharacter_index}, we find that (\ref{eqn:asymp_res}) can be alternatively written as
\begin{equation} \label{eqn:lattice_asymp_blowup}
\inf_{C\in\on{Mov}_1(Y_{\wt{\Sigma}})_\Z}\frac{\pi^*A_\Omega\cdot C}{-K_{\wt{\Sigma}}\cdot C} = \underset{\rho \in \wt{\Sigma}^+}{\on{min}}\big( \frac{\pi^*A_\Omega\cdot C_\rho}{-K_{\wt{\Sigma}}\cdot C_\rho} \big) = \underset{\rho \in \wt{\Sigma}^+}{\on{min}} \frac{\|v_\rho\|^*_\Omega}{1 + v_{\rho,1} + v_{\rho,2}}\end{equation}

Finally, we observe that the $\Omega$-norm is linear on the cone generated by any adjacent pair of rays in $\Sigma$. Therefore, for any ray $\tau \in \wt{\Sigma}^+$, we may write $v_\tau = a \cdot v_\rho + b \cdot v_\sigma$ for $\rho,\sigma \in \Sigma^+$ and positive integers $a,b > 0$, we may apply (\ref{eqn:elementary_sum_fact_1}) to see that
\[
\frac{\|v_\tau\|^*_\Omega}{1 + v_{\tau,1} + v_{\tau,2}} \ge \frac{a\|v_\rho\|^*_\Omega + b\|v_\sigma\|^*_\Omega}{2 + av_{\rho,1} + bv_{\sigma,1} + av_{\rho,2} + bv_{\sigma,2}} \ge \on{min}\Big(\frac{\|v_\rho\|^*_\Omega}{1 + v_{\rho,1} + v_{\rho,2}},\frac{\|v_\sigma\|^*_\Omega}{1 + v_{\sigma,1} + v_{\sigma,2}}\Big)
\]
Thus the minimum in (\ref{eqn:lattice_asymp_blowup}) can be taken over $\Sigma^+$. This concludes the proof. \end{proof}

Finally, the formula in Lemma \ref{lem:toric_asymtotics_polytope} can be generalized to apply to any strongly convex moment domain $\Omega$, to a certain infimum over positive lattice points.

\begin{lemma} \label{lem:toric_asymtotics} Let $\Omega\subseteq\R^2$ be a strongly convex moment domain. Then
\begin{equation} \label{eqn:lattice_asymp_domain}
\lim_{k\to\infty}\frac{\mfk{l}_k(\Omega)}{k}=\lim_{k\to\infty}\frac{\mfk{u}_k(\Omega)}{k}=\inf_{(w_1,w_2)\in\Z^2_{\geq0} \setminus 0}\frac{\|(w_1,w_2)\|_\Omega^*}{1+w_1+w_2}
\end{equation}
\end{lemma}

 \begin{proof} First, assume that $\Omega$ is a rational polytope. Then any vector $(w_1,w_2) \in \Z^2 \setminus 0$ is in the cone generated by two vectors $v_\rho$ and $v_\sigma$ generated by $1$-dimensional cones $\rho,\sigma \in \Sigma$. It follows that
\begin{equation} \label{eqn:lattice_asymp_domain}
\frac{\|(w_1,w_2)\|_\Omega^*}{1 + w_2 + w_2} \ge \on{min}\Big(\frac{\|v_\rho\|^*_\Omega}{1 + v_{\rho,1} + v_{\rho,2}},\frac{\|v_\sigma\|^*_\Omega}{1 + v_{\sigma,1} + v_{\sigma,2}}\Big)
\end{equation}
Thus the righthand side of (\ref{eqn:lattice_asymp_domain}) is bounded below by the righthand side of (\ref{eqn:lattice_asymp}). The bound in the other direction is obvious, so the result follows from Lemma \ref{eqn:lattice_asymp}.

\vspace{3pt}

In the general case, we carry out an approximation argument. Pick a sequence of strongly convex polytopes $\Omega_i$ that satisfy
\[\Omega \subset \Omega_i \subset (1 + 1/i) \cdot \Omega\]
Then by applying the scaling property of the $\Omega$-norm, we have
\[
\lim_{k\to\infty}\frac{\mfk{l}_k(\Omega)}{k} \le \lim_{i \to \infty} \Big(\lim_{k\to\infty}\frac{\mfk{l}_k(\Omega_i)}{k}\Big) = \lim_{i \to \infty} \Big(\inf_{(w_1,w_2)\in\Z^2_{\geq0} \setminus 0}\frac{\|(w_1,w_2)\|_{\Omega_i}^*}{1+w_1+w_2}\Big)\]
\[ \le \lim_{i \to \infty} (1 + 1/i) \cdot \Big(\inf_{(w_1,w_2)\in\Z^2_{\geq0} \setminus 0}\frac{\|(w_1,w_2)\|_{\Omega}^*}{1+w_1+w_2}\Big) = \inf_{(w_1,w_2)\in\Z^2_{\geq0} \setminus 0}\frac{\|(w_1,w_2)\|_{\Omega}^*}{1+w_1+w_2}
\]
The reverse inequality can be proven similarly, and this concludes the argument.\end{proof}

\begin{remark} Much of this discussion generalizes to any dimension if we replace $\mathfrak{u}_k$ with an analogous minimum over the cone of big and movable curves in $Y$ (or of a resolution). However, it is not clear if these higher dimensional asymptotics have any bearing on the asymptotics of the corresponding capacities.\end{remark}

\section{Lower bounds via RSFT axioms} \label{sec:lower_bounds_via_RSFT} In this section, we demonstrate the main lower bound in this paper, Theorem \ref{thm:main_lowerbound}. 

\subsection{Cosphere bundles} \label{subsec:cosphere_bundles} Let $L$ be a closed manifold with a metric $g$. The unit cosphere bundle $S^*L \subset T^*L$ of $L$ is a contact manifold with contact form $\alpha = \lambda|_L$ given by the restriction of the standard Liouville form. The Reeb orbits $\gamma$ of $S^*L$ are in bijection with the geodesics of $(L,g)$.

\vspace{3pt}

A Reeb orbit in this setting has three associated integer invariants: the Maslov, Morse and Conley-Zehnder indices. 

\begin{definition} \label{def:Maslov} The \emph{Maslov index} $\mu(\gamma,\tau)$ of $\gamma$ is the Maslov index of the loop of Lagrangians
\[L_\tau \qquad\text{given by}\qquad L_\tau(t) := \tau_{\gamma(t)}(T_{\gamma(t)}L) \subset \C^n\]\end{definition}

\begin{definition} \label{def:Morse} The \emph{Morse index} $I(\gamma)$ is the Morse index of the corresponding geodesic $\pi \circ \gamma:S^1 \to L$ as a critical points of the energy functional $E_g$ on the free loop space $\Lambda L$.
\end{definition}

\noindent The Conley-Zehnder index $\on{CZ}(\gamma,\tau)$ is defined in \cite{v1990} in the general case, i.e. for Reeb orbits whose linearized return map has a $1$-eigenvalue. We will not review the definition here.

\vspace{3pt}

These three invariants are related by the following simple formula. Note that this result does not require any non-degeneracy hypotheses on the orbit.

\begin{lemma}[Viterbo, \cite{v1990}] \label{lem:Morse_vs_CZ} Let $\gamma$ be a Reeb orbit given by the lift of a closed geodesic $q$. Then
\[\on{CZ}(\gamma,\tau) + \mu(\gamma,\tau) = I(\gamma)\] 
\end{lemma}

Turning to the relevant situation for this paper, we consider the torus $T^n$ with the standard flat metric and the corresponding cosphere bundle $S^*T^n$. 

\begin{lemma} \label{lem:geodesic_Morse_torus} The Morse index $I(\gamma)$ of any Reeb orbit in $S^*T^n$ is zero.
\end{lemma}

\begin{proof} Let $q = \pi \circ \gamma$ be the geodesic corresponding to $\gamma$. The Morse index theorem \cite[Thm 2.5.14]{k2011} states that
\begin{equation}
I(\gamma) = \# \{\text{conjugate points of }q\} + \text{concavity}(q)
\end{equation}
The geodesics of a metric of nowhere positive curvature has no conjugate points \cite[Thm 2.6.2]{k2011}. Furthermore, the concavity satisfies
\[
\text{concavity}(q) \le n - 1 - \text{null}(q)
\]
Here $\on{null}(q)$ is one less than the dimension of the null-space of the Hessian of the energy functional $E_g$ on the free loop space $\Lambda T^n$. The geodesics of $T^n$ come in $(n-1)$-dimensional Morse-Bott families, and thus $\on{null}(q) = 0$. Therefore
\[
I(\gamma) = \text{concavity}(q) \le n - 1 - \text{null}(q) = 0
\] 
Since any Morse index is non-negative, we acquire the desired result.\end{proof}

\subsection{Free domains} \label{subsec:free_domain} We now discuss the implications of \S \ref{subsec:cosphere_bundles} for free toric domains. Fix a smooth convex domain in $\R^n$ containing $0$.
\[\Omega \subset \R^n \qquad \text{with}\qquad 0 \in \R^n\]
This convex domain determines a Liouville domain within $T^*T^n$, given as a subset of $T^n \times \R^n$ by
\[D_\Omega^*T^n := T^n \times \Omega \qquad\text{and}\qquad S^*_\Omega T^n := T^n \times \partial\Omega\]

The Reeb orbits on $S^*_\Omega T^n$ are in bijection with the geodesics of the unique Finsler metric on $T^n$ with unit disk bundle $D_\Omega^*T^n$. In fact, we have
\begin{lemma} \label{lem:mb_family_orbits_torus} The (unparametrized) Reeb orbits of $S^*_\Omega T^n$ come in Morse-Bott families
\[S_a(\Omega) \simeq T^{n-1} \qquad\text{for}\qquad a \in H_1(T^n;\Z)\text{ with }a \neq 0\]
\end{lemma}
\noindent Providing an explicit description of the orbits in these Morse-Bott families is simple. To start, choose a non-zero homology class and point
\[a \in H^1(T^n;\Z) \simeq \Z^n \qquad\text{and}\qquad q \in T^n\]
Let $L = \|a\|^*_\Omega$ and let $p \in \partial \Omega$ satisfy $\langle p,a\rangle = \|a\|^*_\Omega$. Then there is a Reeb orbit of period $L$
\[\gamma_{a,q}:\R/L\Z \to S^*_\Omega L \quad\text{given by}\quad \gamma_{a,q}(t) = (q + ta/L, p)\]
Note that the homology class of $\gamma_{a,q}$ is $a$ and that the orbits $\gamma_{a,q}$ for a fixed $a$ and varying point $p$ form a Morse-Bott $T^n$-family of parametrized orbits. Quotienting by the parametrization yields the $T^{n-1}$-family of Lemma \ref{lem:mb_family_orbits_torus}.

\vspace{3pt}

The Conley-Zehnder indices of these Morse-Bott families can be calculated as follows. 

\begin{lemma} \label{lem:CZ_of_MB_families} Let $\Gamma$ be a set of Reeb orbits of $S^*_\Omega T^n$ bounding an immersed surface $\Sigma \subset D^*_\Omega T^n$ and let $\tau$ be a trivialization of $\xi$ over $\Gamma$ that extends to a trivialization of $T(D^*T^n)$ over $\Sigma$. Then
\[\on{CZ}(\Gamma,\tau) = 0\]
\end{lemma}

\begin{proof} First assume that $\Omega = B^n \subset \R^n$ is the unit $n$-ball. Then $S^*_\Omega T^n$ is just the standard unit cosphere bundle of $T^n$. By Lemma \ref{lem:Morse_vs_CZ}, we see that
\[\on{CZ}(\Gamma,\tau) + \mu(\Gamma,\tau) = I(\Gamma) \]
The family of Lagrangians $L_\tau$ over $\Gamma$ in Definition \ref{def:Maslov} extends to a family parametrized by $\Sigma$, so $\mu(\Gamma,\tau) = \mu(L_\tau) = 0$. Furthermore, $I(\Gamma) = 0$ by Lemma \ref{lem:geodesic_Morse_torus}. Thus $\on{CZ}(\Gamma,\tau) = 0$.

For the general case of any convex domain $\Omega$, let $\Gamma = (\gamma_1,\dots,\gamma_m)$ and let $a_i = [\gamma_i] \in H_1(T^n;\Z)$. Let $\Omega_t$ be a family of convex domains with $\Omega_0 = \Omega$ and $\Omega_1 = D^n$. We have $1$-parameter families of Morse-Bott $T^{n-1}$-families of Reeb orbits
\[S_{a_i}(\Omega_t) \subset S^*_{\Omega_t}T^n \qquad\text{where}\qquad a \in H^1(T^n;\Z)\text{ and }t \in [0,1]\]
We can choose a $1$-parameter family of orbit sets $\Gamma_t$ with component orbits $\gamma_{i,t}$ that satisfy
\[\gamma_{i,t} \in S_{a_i}(\Omega_t) \qquad\text{and}\qquad \gamma_{i,0} = \gamma_i\]
We may also choose a family of trivializations $\tau_t$ over $\Gamma_t$ with $\tau_0 = \tau$. Associated to these choices is a $1$-parameter family of linearized Reeb flows  
\[
\Phi_\tau:[0,1]_t \times [0,1]_s \to \on{Sp}(2n-2) \quad\text{where}\quad \on{rank}(\Phi_\tau(t,1) - \on{Id}) \text{ is constant in }t
\]
The Conley-Zehnder index is unchanged under homotopies of paths in the symplectic group where the rank on the righthand side is constant. Therefore
\[\on{CZ}(\Gamma_t,\tau_t) \quad\text{is constant in }t\]
Thus we are reduced to the case where $\Omega$ is the unit $n$-ball.\end{proof}

\subsection{Proof of Theorem \ref{thm:main_lowerbound}} \label{subsec:proof_of_lower_bound} We now apply the analysis of \S \ref{subsec:cosphere_bundles} and \ref{subsec:free_domain} to prove the lower bound Theorem \ref{thm:main_lowerbound} from the introduction.

\vspace{3pt}

We will need to perturb the Morse-Bott contact form on the free domains $D^*_\Omega T^n$ to contact forms that are Morse (i.e. non-degenerate) below a certain action threshold. 

\begin{lemma}[Morse Perturbation, cf. \cite{b2002}] \label{lem:morse_perturbation} Let $(Y,\alpha)$ be a contact manifold with Morse-Bott Reeb orbits. Fix $\epsilon > 0$ and $L > 0$. Then there exists a non-degenerate contact form $\alpha_\epsilon$ with the following properties.
\begin{itemize}
	\item[(a)] (Closeness) $\alpha_\epsilon$ is uniformly close to $\alpha$, i.e.
	\[\|\alpha_\epsilon - \alpha\|_{C^0} < \epsilon\]
	\item[(b)] (Reeb Orbits) If $\gamma_\epsilon$ is a Reeb orbit of $\alpha_\epsilon$ with action less than $L$, then $\gamma_\epsilon$ is non-degenerate and there exists a Reeb orbit $\gamma$ in a Morse-Bott family $S$ of $\alpha$ such that
	\[\|\gamma - \gamma_\epsilon\|_{C^2} < \epsilon \qquad\text{and}\qquad \|\on{CZ}(\gamma_\epsilon,\tau) - \on{CZ}(\gamma,\tau)\| \le \on{dim}(S)\]
\end{itemize}
Here $\tau$ is a trivialization of $\xi$ along $\gamma$ (extended to a neighborhood of $\gamma$, and thus to $\gamma_\epsilon$).
\end{lemma}  

In the smooth case, the main lower bound follows quickly from the following lemma.

\begin{lemma} \label{lem:codim_vs_punctures} Let $\Omega \subset \R^n$ be a convex domain with smooth boundary and let $P$ be a tangency constraint. Then there exists an orbit set $\Gamma$ on $S^*_\Omega T^n$ such that
\[[\Gamma] = 0 \in H_1(T^n) \qquad \on{codim}(P) \le (2n - 2) \cdot (\#\Gamma - 1) \qquad \mathcal{A}(\Gamma) \le \mathfrak{r}_P(D^*_\Omega T^n)\] 
\end{lemma}

\begin{proof} First, assume that $\Omega$ has smooth boundary. Let $\epsilon > 0$ and $L = 2 \cdot \mathfrak{r}_P(D^*_\Omega T^n)$, and take a Morse perturbation $\alpha_\epsilon$ of $\alpha$ as in Lemma \ref{lem:morse_perturbation}. By the Reeb orbit axiom of $\mathfrak{r}_P$ (i.e. Theorem \ref{thm:main_rP_axioms}(c)), there is a list of Reeb orbits
\[\Gamma_\epsilon = (\gamma^1_\epsilon,\dots,\gamma^m_\epsilon)\]
of $\alpha_\epsilon$ and a (possibly disconnected) genus $0$ immersed surface $\Sigma \subset X$ bounding $\Gamma$ such that
\[
	 \mathcal{A}(\Gamma) = \sum_{\gamma \in \Gamma} \int_\gamma \lambda \le \mathfrak{r}_P(X) \qquad\text{and}\qquad \on{ind}(\Sigma) = \on{codim}(P)
\] 
Applying the formula for the Fredholm index of $\Sigma$ using a trivialization $\tau$ that extends to all of $\Sigma$, we acquire the following.
\[
\on{codim}(P) = (n - 3) \cdot \chi(\Sigma) + \on{CZ}(\Gamma_\epsilon,\tau)
\]
Now we estimate the Euler characteristic and Conley-Zehnder index. 

\vspace{3pt}

For the Euler characteristic, let $l$ be the number of components of $\Sigma$. Since $\Sigma$ is genus $0$ and has punctures on the orbit set $\Gamma_\epsilon$, the Euler characteristic is given by
\[
\chi(\Sigma) = 2l - m 
\]
For the Conley-Zehnder index, we note that by Lemma \ref{lem:morse_perturbation}(b), the orbit set $\Gamma_\epsilon$ is $C^2$-close to an orbit set of $\alpha$ of the form
\[\Gamma = (\gamma^1,\dots,\gamma^m) \qquad\text{with}\qquad [\Gamma] = [\Gamma_\epsilon] \quad\text{and}\quad |\on{CZ}(\Gamma,\tau) - \on{CZ}(\Gamma_\epsilon,\tau)| \le m \cdot (n-1)\]
Since $\Gamma_\epsilon$ bounds a surface, so does $\Gamma$. Thus by Lemma \ref{lem:CZ_of_MB_families}, we have $\on{CZ}(\Gamma,\tau) = 0$. Therefore
\[\on{CZ}(\Gamma_\epsilon,\tau) \le m \cdot (n-1)\]
Adding together these two estimates, we acquire the following bound.
\begin{equation}\label{eqn:codim_vs_punctures_1} \on{codim}(P) \le (n-3) \cdot (2l - m) + (n-1) \cdot m = 2m + (2n - 6) \cdot l\end{equation}

Now observe that every component of $\Sigma$ must have at least two punctures. Indeed, $T^*T^n$ is exact, so there are no closed surfaces of positive area and no component can have $0$ punctures. Moreover, no Reeb orbit in $S^*_\Omega T^n$ is null-homotopic, so there cannot be a component with $1$ puncture. Thus
\begin{equation}\label{eqn:codim_vs_punctures_2} l \le m/2 \quad\text{and}\quad m \ge 2 \qquad\text{where}\qquad m = \# \Gamma\end{equation}
Now we can apply (\ref{eqn:codim_vs_punctures_2}) to bound the righthand side of (\ref{eqn:codim_vs_punctures_1}) by case on $n$. For example
\[2m + (2n - 6) l \le 2m + (n - 3) m \le (n - 1) m \le (2n - 2)(m-1) \qquad \text{if} \qquad n \ge 4 \]
The other cases of $n = 2$ and $n = 3$ can be checked similarly. This concludes the proof.
\end{proof}

Theorem \ref{thm:main_lowerbound} is an easy corollary of Lemma \ref{lem:codim_vs_punctures} and a simple limiting argument. 

\begin{cor} \label{cor:main_lower_bound} Let $\Omega \subset \R^{2n}$ be a weakly convex domain. Then
\[\mathfrak{l}_k(\Omega) \le \mathfrak{r}_P(D^*_\Omega T^n) \le \mathfrak{r}_P(X_\Omega) \qquad\text{where}\qquad 2k = \on{codim}(P)\]
\end{cor}

\begin{proof} Let $\Delta \subset \Omega$ be a domain acquired by shrinking $\Omega$ slightly and then translating it to be contained in $\Omega$. Then since the boundary of $\Delta$ does not touch the boundary of $[0,\infty)^n$, we have a toric symplectomorphism
\[
X_\Delta \simeq D^*_\Delta T^n \quad\text{and thus}\quad \mathfrak{r}_P(D^*_\Delta T^n) \le \mathfrak{r}_P(X_\Omega) 
\]
In the limit as $\Delta$ approaches $\Omega$ in the $C^0$-topology, we acquire $\mathfrak{r}_P(D^*_\Omega T^n) \le \mathfrak{r}_P(X_\Omega)$.

\vspace{3pt}

To lower bound $\mathfrak{r}_P(D^*_\Omega T^n)$, let $\Gamma$ be a Reeb orbit set with component orbits $\gamma_i$ in homology class $[\gamma_i] = v_i$ in $H_1(T^2;\Z) \simeq \Z^2$. Then
\begin{equation} \label{eqn:formulas_for_action_and_homology}
\mathcal{A}(\Gamma) = \sum_i \mathcal{A}(\gamma_i) = \sum_i \|v_i\|_{\Omega}^* \qquad [\Gamma] = \sum_i v_i \in \Z^2\end{equation}
Thus, the lower bound of $\mathfrak{r}_P(D^*_\Omega T^n)$ follows immediately from (\ref{eqn:formulas_for_action_and_homology}) and Lemma \ref{lem:codim_vs_punctures} when $\Omega$ is smooth. In general, we can approximate $\Omega$ in $C^0$ by a family of smooth domains $\Omega(\epsilon)$ with
\[(1 - \epsilon) \cdot \Omega(\epsilon) \subset \Omega \subset (1 + \epsilon) \cdot \Omega(\epsilon)\]
The Monotonicity axiom of $\mathfrak{r}_P$ (i.e. Theorem \ref{thm:main_rP_axioms}(a)) implies that
\[\mathfrak{l}_k(\Omega(\epsilon)) \to \mathfrak{l}_k(\Omega) \qquad\text{and}\qquad \mathfrak{r}_P(D_{\Omega(\epsilon)}^*T^n)) \to \mathfrak{r}_P(D_\Omega^*T^n)\] 
Thus the singular case follows from the smooth case. \end{proof}

\section{Upper bounds via movable curves} \label{subsec:upper_bounds_via_movable_curves}

In this section, we demonstrate the main upper bounds in this paper, Theorems \ref{thm:main_upperbound}, \ref{thm:main_upperbound_closed} and \ref{thm:main_upperbound_stable}. As discussed in the introduction, the main tool is a construction of immersed symplectic spheres out of cocharacter curves. 

\subsection{Immersed caps for torus knots} \label{subsec:torus_knots} To explain our proof, we need some preliminaries on the symplectic topology of transverse torus knots. 

\vspace{3pt}

Fix coprime positive integers $p$ and $q$. We let $V(p,q) \subset \C^2$ and $L(p,q) \subset S^3$ denote the $(p,q)$-singularity and $(p,q)$-torus knot respectively.
\begin{equation} V(p,q) := \{(x,y) \in \C^2 \; : \; x^q - y^p = 0\} \qquad L(p,q) := V(p,q) \cap S^3\end{equation}
Note that we can take $S^3 \subset \C^2$ to be the sphere of any radius centered at $0$, since $V(p,q)$ is transverse to every such sphere. It is well known that $L(p,q)$ is a transverse knot with respect to the standard contact structure induced on $S^3$ by restricting the radial Liouville form of $\C^2$ to $S^3$. 

\vspace{3pt}

In fact, we may view $L(p,q)$ as a braid using the standard open book on $S^3$ with binding $B$ and projection $\pi$ given by
\[
B := (\C \times 0) \cap S^3 \quad\text{and}\quad \pi:S^3 \setminus B \to S^1 \text{ defined as }\pi(x,y) = \frac{y}{|y|} \in S^1 \subset \C
\]

\begin{lemma} The knot $L(p,q)$ is in braid position with respect to the standard open book on $S^3$. \end{lemma}

\begin{proof} A knot is in braid position with respect to an open book if it is disjoint from the binding and transverse to the pages. Using the sphere of radius $\sqrt{2}$, we can parametrize $L(p,q)$ as so.
\[
\gamma:S^1 = U(1) \to \C^2 \qquad \gamma(t) = (t^p,t^q)
\]
We see that this curve is disjoint from $\C \times 0$ and that $\pi \circ \gamma:U(1) \to U(1)$ is the map $\pi \circ \gamma(t) = t^q$. This implies the lemma. \end{proof}

We will need a family of immersed disks for use in later sections. To construct these disks, we start by showing that every $(p,q)$-knot is a regularly and transversely homotopic to the unknot.

\begin{lemma} \label{lem:homotopy_of_Lpq} There is a regular homotopy $K_t$ of transverse knots from the standard transverse unknot $U$ to $L(p,q)$, that only has positive self-intersections.
\end{lemma}

\begin{proof} Consider the standard braid $B(p,q)$ representing $L(p,q)$, written in terms of (right-handed) braid generators $\sigma_i$ of the $p$th braid group $B_p$ as
\[
B(p,q) = g^q \qquad g := \sigma_1\sigma_2 \dots \sigma_{p-1}
\]
We may describe $B(p,q)$ topologically as so. Order and label the strands of the braid as $s_1,s_2,\dots,s_p$ from top to bottom. Then $g$ takes the top strand of the braid and crosses it positively over the other strands to move it to the bottom, introducing $p-1$ crossings. $B(p,q)$ does this $q$ times, introducing $q(p-1)$ crossings. 

\begin{figure}[h]
\centering
\caption{The $B(p,q)$ braid for $p = 5$ and $q = 4$.}
\includegraphics[width=.6\textwidth]{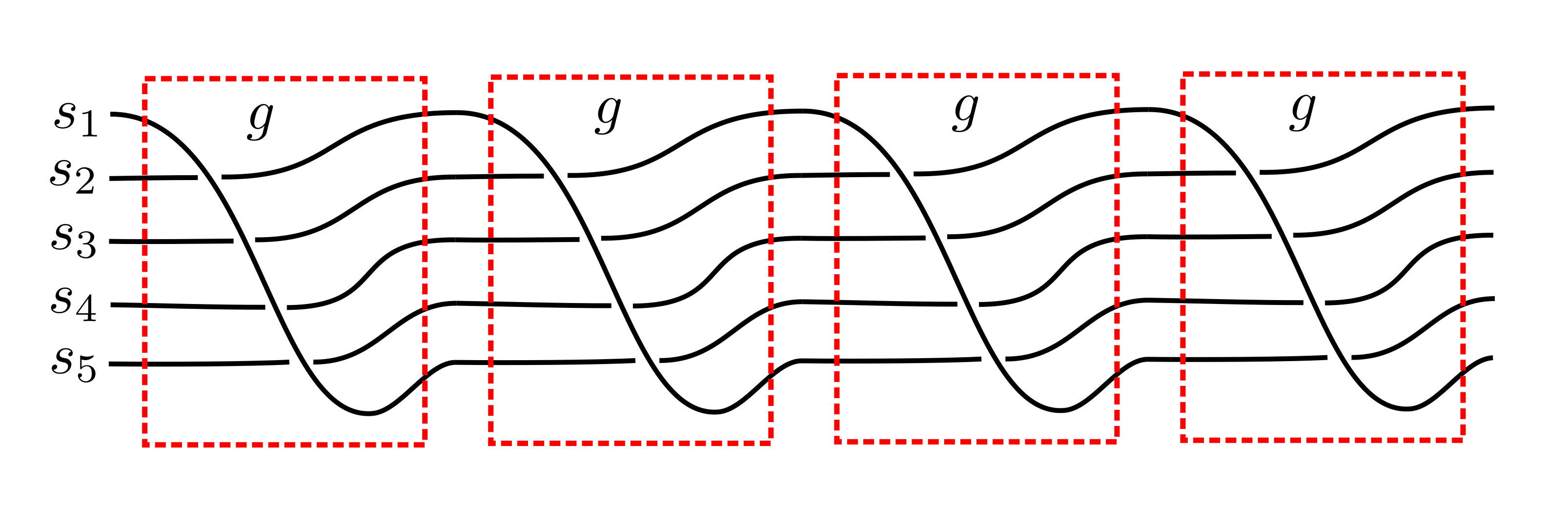}
\label{fig:pq_braid}
\end{figure}

\begin{figure}[h]
\centering
\caption{The $C$ braid for $p = 5$ and $q = 4$.}
\includegraphics[width=.6\textwidth]{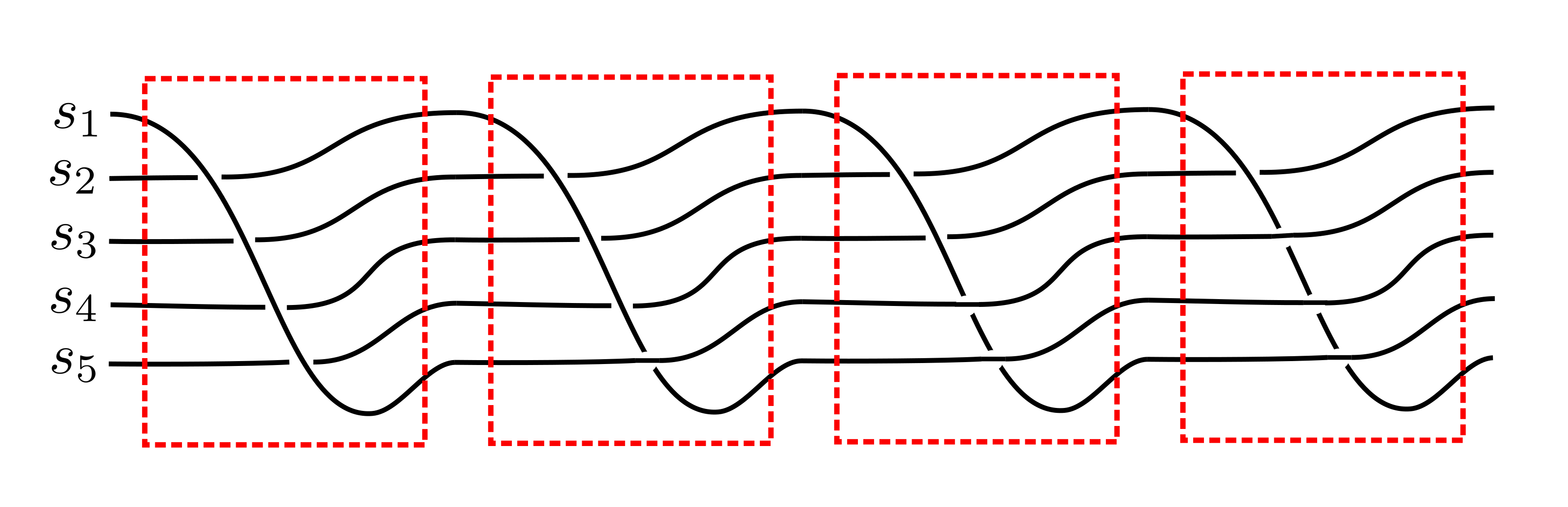}
\label{fig:Upq_braid}
\end{figure} 

We may form a new braid $C$ by changing each positive (right-hand) crossing of $s_i$ over $s_j$ with $i > j$ to a negative (left-hand) crossing. This braid is ascending. Therefore the closure $U$ of $C$ is a topological unknot, and we must compute its self-linking number. Given a braid $B$, let $n_\pm(B)$ denote the number of $\pm$ crossings and $b(B)$ be the number of strands. Then the self-linking number of the braid closure is
\[\on{sl}(K) = n_+(B) - n_-(B) + b(B)\]
In the case of the braid $C$, one can easily compute each of the above numbers.
\[
n_+(C) = (p-1)q - \frac{1}{2}(p-1)(q-1) \qquad n_-(C) = \frac{1}{2}(p-1)(q-1) \qquad b(C) = p
\]
 Therefore $\on{sl}(U) = -1$, so $U$ is the standard unknot of self-linking number $-1$.

\vspace{3pt}

The regular homotopy $K_t$ can be realized by taking the braid closure and performing the homotopy corresponding to the crossing changes from $C$ to $B(p,q)$. Since these are all $-$ to $+$ crossing changes, the corresponding intersections in the homotopy are positive. \end{proof}

Any regular homotopy can be viewed as an immersed symplectic cobordism in the symplectization. Here is the precise version of this statement that we need.

\begin{lemma}[Homotopy to Cobordism] \cite[Lem 2.4]{eg2020}  Let $\phi:[0,1] \times S^1 \to M$ be a regular homotopy between transverse links through a contact $3$-manifold $M$. Then for $a > 0$ sufficiently large, the trace map
\[
\Phi:[0,1] \times S^1 \to (-\infty,0] \times M \qquad \Phi(t,s) := ((a + 1)t - a,\phi_t(s))
\]
is a symplectic immersion. Furthermore, every $\pm$-double point of $\phi$ becomes a $\pm$-double point of $\Phi$.
\end{lemma}

The standard unknot bounds a symplectic disk $D$ in $S^3$ (and therefore $B^4$). By gluing this disk $D$ to the trace surface $\Sigma$ of the regular homotopy in Lemma \ref{lem:homotopy_of_Lpq}, we have the following corollary.

\begin{cor} \label{cor:capping_torus_knots} The transverse knot $L(p,q) \subset S^3$  bounds an immersed symplectic disk $\Sigma \subset B^4$ with only positive self-intersections. 
\end{cor}

\subsection{Symplectic spheres representing cocharacters} \label{subsec:symplectic_spheres} We next apply the discussion in \S \ref{subsec:torus_knots} to find well-behaved immersed symplectic spheres that represent cocharacters in homology.

\vspace{3pt}

We first need some observations about the curve $C_\rho$ for any $\rho\in\Sigma_\Omega(1)$. . The map $u_\rho\colon\pr^1\to Y_\Sigma$ defining $C_\rho$ is injective and restricts to a holomorphic group map $u_\rho:\C^\times \to (\C^\times)^2 \subset Y$. There are two points $0_\rho$ and $\infty_\rho$ of $C_\rho$ in the complement of $u(\C^\times)$. 

\begin{figure}[h]
\centering
\caption{A depiction of the cocharacter curve $C_\rho$.}
\includegraphics[width=.3\textwidth]{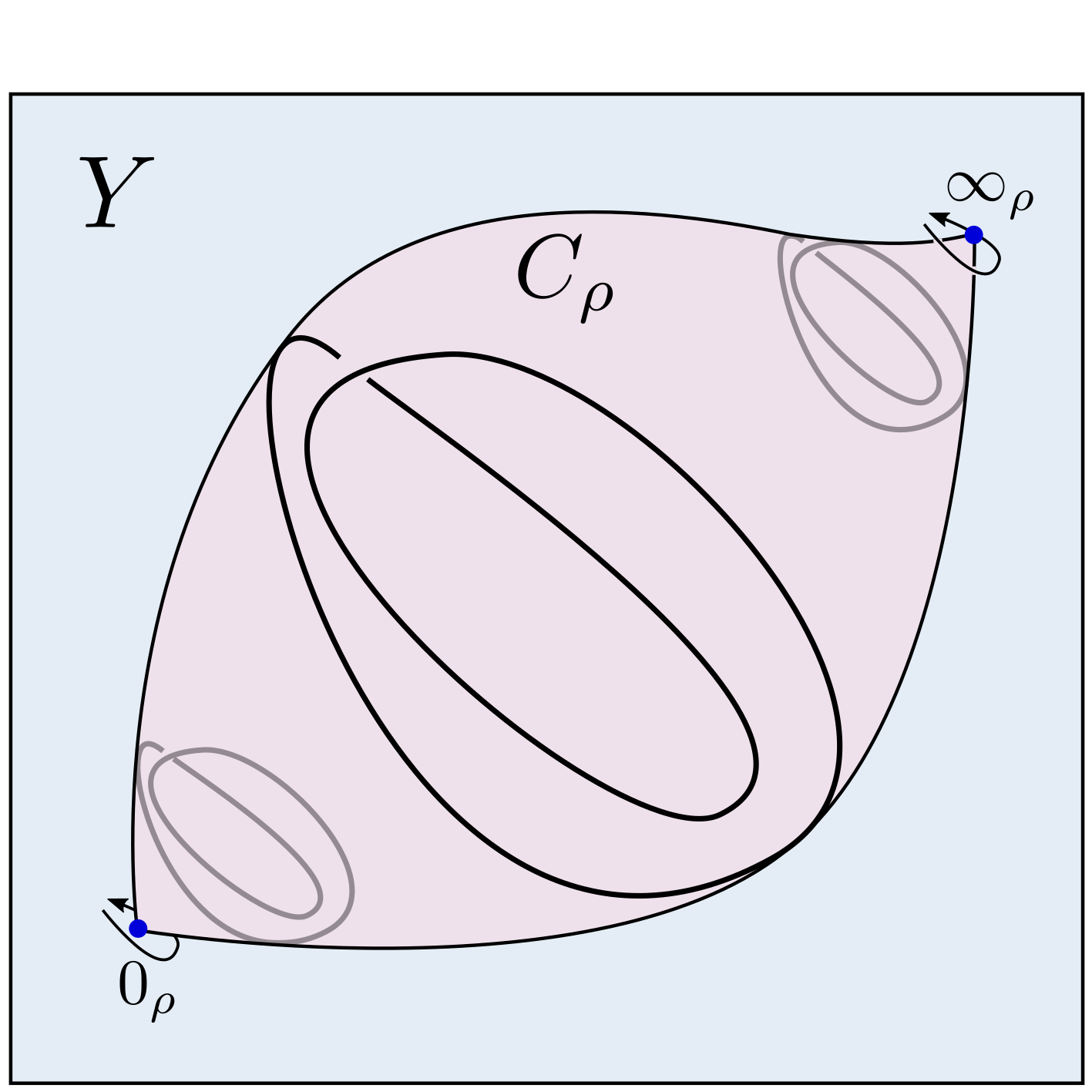}
\label{fig:cocharacter_curve}
\end{figure}

The singularities of $C_\rho$ are both (or one) of the points $0_\rho$ and $\infty_\rho$, and when one of these points is singular then it must be equal to a torus fixed point $p$ of $Y$. If we choose a toric affine chart $\C^2 \simeq U \subset Y$ centered at $p$ and $p \in C_\rho$, then there is a coprime $a$ and $b$ such that
\begin{equation}u_\rho(t) = (t^a,t^b) \qquad\text{and}\qquad C_\rho \cap \C^2 = V(a,b) \qquad \text{in the chart $U \simeq \C^2$} \end{equation}
Note that in the ball $B$ of radius $1$ in the chart $U$ above, $C_\rho \cap C_\sigma \cap B = p$ if $p \in C_\rho \cap C_\sigma$. Furthermore, $C_\rho \cap C_\sigma \cap B$ is empty when $p \not\in C_\rho \cap C_\sigma$.

\begin{lemma} \label{lem:cocharacter_spheres} Let $P$ be a finite set of distinct primitive vectors in the cocharacter lattice of $Y$. Then there is a collection of immersed symplectic spheres $S_\rho$ for $\rho \in P$ with $[S_\rho] = [C_\rho] \in H_2(Y)$ that satisfy
\begin{itemize}
	\item[(a)] All self-intersections of $S_\rho$ are positive, transverse double points .
	\item[(b)] All interections in $S_\rho \cap S_\sigma$ are positive, transverse double points for each distinct pair $\rho,\sigma \in P$.
\end{itemize}
\end{lemma}

\begin{proof} We construct the immersed symplectic spheres $S_\rho$ by modifying $C_\rho$ in a sequence of steps.

\vspace{3pt}

{\bf Step 1.} We first modify the cocharacter closure curves $C_\rho$ to have disjoint singularities. Let $y \in Y$ be a torus fixed point and let $P(y) \subset P$ be the set of cocharacter elements in $P$ with $p \in C_\rho$. If $D \subset \C$ denotes a disk of sufficiently small radius, then
\[
u_\rho(D) \cap u_\sigma(D) = 0 \in \C^2 \text{ if }\rho \neq \sigma \in P_y \quad\text{and}\quad u_\rho(D) \cap u_\sigma(D) = \emptyset \text{ if }\rho \in P_y \text{ and }\sigma \not\in P_y
\]
Now for each $\rho \in P_y$, choose a $C^\infty$-small smooth map $\epsilon_\rho:\C \to \C^2$ such that
\[
\epsilon_\rho = 0 \text{ on }\C \setminus D \qquad \epsilon_\rho = c_\rho \in \C^2 \text{ near } 0 \quad\text{and}\quad c_\rho \neq c_\sigma \text{ if }\rho \neq \sigma
\]
Then let $v_\rho = u_\rho + c_\rho$ and let $C'_\rho$ be given by $v_\rho(\C)$ near $0$ and $C_\rho$ elsewhere.  

\vspace{3pt} 

By performing this modification near every torus fixed point of $Y$, we can produce a set $C'_\rho$ of (singular) surfaces where the singularities of $C'_\rho$ and $C'_\sigma$ are distinct for distinct vectors $\rho$ and $\sigma$ in $P$. By choosing $\epsilon_\rho$ $C^\infty$-small enough we may also assume that the intersection $C'_\rho \cap C'_\sigma$ is contained in a region were both $C'_\rho$ and $C'_\sigma$ are smooth holomorphic sub-manifolds.

\vspace{3pt}

{\bf Step 2.} Next we remove the singularities of $C'_\rho$ for each $\rho \in P$ to produce a smooth immersion of a sphere $S_\rho$. Let $p \in C'_\rho$ be a singular point. By Step 1, we may choose a neighborhood $U$ of $p$ and a holomorphic identification $U \simeq B^4$ such that
 \[C'_\rho \cap U \simeq V(p,q) \subset B^4 \cap B \qquad C'_\sigma \cap B = 0 \text{ if }\sigma \neq \rho\]
 Thus $C'_\rho \cap \partial U$ is a transverse knot equivalent to $L(p,q)$ under a contactomorphism $\partial U \simeq \partial B^4 = S^3$. Thus by Corollary \ref{cor:capping_torus_knots}, $C'_\rho \cap \partial U$ bounds an immersed symplectic disk $D(\rho,p)$ in $U$ with only positive self-intersections. We may thus define $S_\rho$ by the property that
 \[
 S_\rho = D(\rho,p) \text{ near any singular }p \in C'_\rho \quad\text{and}\quad S_\rho = C'_\rho \text{ elsewhere}
 \]

 {\bf Step 3.} Finally, we perturb the symplectic submanifolds 
$S_\rho$ so that $S_\rho \pitchfork S_\sigma$ for each $\rho \neq \sigma$. Since $S_\rho$ and $S_\sigma$ are holomorphic sub-manifolds in a neighborhood of the set $S_\rho \cap S_\sigma$, the resulting intersections are positive.
\end{proof}

\begin{prop} \label{prop:immersed_sphere} Let $Y$ be a smooth toric surface and let $Z$ be a curve class of the form
\[Z = \sum_\rho k_\rho C_\rho \qquad\text{with}\qquad k_\rho \ge 0\quad\text{and}\quad m = \sum_\rho k_\rho\]
Then there exists a symplectic immersion of the form
\[\iota:\Sigma \to Y \qquad \text{with}\quad \Sigma := \sqcup_1^m S^2 \quad\text{and} \quad \iota^*[\Sigma] = [Z]\]
that has only positive, double-point self-intersections.
\end{prop}

\begin{proof} Let $P$ be the finite set of cocharacters $\rho$ such that $k_\rho > 0$. Choose a set of immersed spheres $S_\rho$ for each $\rho \in P$ that satisfy the conditions of Lemma \ref{lem:cocharacter_spheres}(a)-(b). Form a collection $\mathcal{S}$ of $m$ spheres in $Y$ by taking a union of $k_\rho$ copies of $S_\rho$ for each $\rho \in P$, and define a map
\[
\jmath:\Sigma \to Y \qquad \text{with}\qquad \jmath_*[\Sigma] = [Z]
\]
by parametrizing each sphere $S \subset Y$ in $\mathcal{S}$ and then perturbing the map to have transverse self-intersections. Note that the Euler characteristic of the normal bundle $\nu(S_\rho)$ is non-negative. Indeed, we have
\[
\chi(\nu(S_\rho)) \cdot  = c_1(Y) \cdot \Sigma - \chi(S^2) = -K \cdot C_\rho - 2 \ge 0
\] 
Therefore the perturbation in the construction of $\jmath$ can be chosen to that the self-intersections of the map are all positive double points. 
\end{proof}

\begin{lemma}[Surgery] \cite[Lem 2.6]{eg2020} \label{lem:surgery} Let $\iota:\Sigma \to X$ be a symplectic immersion of a $2$-manifold $\Sigma$ into a symplectic $4$-manifold $X$. Let $p,q \in \Sigma$ map to a positive, transverse double point $\iota(p) = \iota(q)$. 

\vspace{3pt}

Then there is a symplectic immersion $\tilde{\iota}:\tilde{\Sigma} \to X$ from the $1$-surgery $\tilde{\Sigma}$ of $\Sigma$ along the $0$-sphere $p \cup q$ to $X$ that agrees with $\Sigma$ outside of a small neighborhood of $p$ and $q$.
\end{lemma}

\begin{cor} \label{cor:many_spheres_to_one_sphere} Let $\Sigma = \sqcup_1^m S^2 $ and $\jmath:\Sigma \to X$ be a symplectic immersion of $m$ spheres into a symplectic $4$-manifold $X$ with only transverse double points. Suppose that $\jmath(\Sigma)$ is connected.

\vspace{3pt}

Then there is a symplectic  immersion $\iota:S^2 \to X$ i with only transverse double points. \end{cor}

\begin{proof} Perform $m-1$ surgeries (via Lemma \ref{lem:surgery}) on the double points of the map $\jmath$, and on each surgery choose a pair of surgery points on different components of the domain. This will produce the desired map $\iota$. \end{proof}

\subsection{Proof of Theorems \ref{thm:main_upperbound}, \ref{thm:main_upperbound_closed} and \ref{thm:main_upperbound_stable}} \label{subsec:proof_of_thm_closed_and_stable} We now apply the analysis of \S \ref{subsec:torus_knots} and \ref{subsec:symplectic_spheres} to prove Theorems \ref{thm:main_upperbound_closed} \ref{thm:main_upperbound}, and \ref{thm:main_upperbound_stable} from the introduction.

\vspace{3pt}

We start by noting that the $+$ movable curves have non-trivial Gromov--Witten invariants. This uses the Immersed Spheres property (Lemma \ref{lem:immered_spheres}).

\begin{lemma} \label{lem:GW_of_movable_curve} Let $\Omega \subset \R^2$ be rational polytope that is strongly convex and smooth. Let $P$ be a lax tangency constraint and let $A \in \on{Mov}^+(Y_\Sigma)$. Then
\[
\on{GW}_P(Y_\Sigma,A) > 0
\]
\end{lemma}

\begin{proof} The movable cone is generated by the cocharacter curves by Lemma \ref{lem:strongly_convex_movable}. Thus we write
\[C = \sum_{\rho \in \Sigma^+} k_\rho \cdot C_\rho \qquad\text{where}\qquad C_\rho \ge 0\]
There are two cases on $C$: either $C$ is represented by a single cocharacter with $0$ self-intersection or $C$ has positive self-intersection.

\vspace{3pt}

First, if $C = [C_\rho]$ where $[C_\rho] \cdot [C_\rho] = 0$, then Proposition \ref{prop:immersed_sphere} produces a single immersed sphere $\Sigma$ in $Y_\Sigma$ with positive self-intersections. Since $P$ is lax, we can choose a set of $\# P$ points on $\Sigma$ and isotope $\Sigma$ to satisfy $P$ at those points. We then apply Lemma \ref{lem:immered_spheres} to deduce the result.

\vspace{3pt}

Second, if $C$ satisfies $C \cdot C > 0$ then either $k_\rho > 0$ for more than two rays $\rho$, or $k_\rho > 0$ for a single $\rho$ with $[C_\rho] \cdot [C_\rho] > 0$. Either way, Proposition \ref{prop:immersed_sphere} produces an immersion of spheres with connected image and positive self-intersections. We plumb this collection by Corollary \ref{cor:many_spheres_to_one_sphere} to acquire a single immersed sphere. Then we apply Lemma \ref{lem:immered_spheres}.
\end{proof}

\begin{remark} We note that this argument breaks down for classes of the form $C = k_\rho \cdot [C_\rho]$ with $k_\rho \ge 2$ and $[C_\rho] \cdot [C_\rho] = 0$. In fact, when $P = ((k))$ is the $1$-point, $k$-fold tangency condition the Gromov--Witten invariants are known to vanish.  
\end{remark}

Theorem \ref{thm:main_upperbound_closed}/\ref{thm:main_upperbound_closed_body} is now an easy application of the Gromov--Witten axiom of the RSFT capacities.

\begin{thm} \label{thm:main_upperbound_closed_body} Let $X \to Y_\Sigma$ be a symplectic embedding of a star-shaped domain $X$ into a strongly convex toric surface $Y_\Sigma$ with symplectic form $\omega_\Omega$. Then
\begin{equation} \label{eqn:main_upperbound_closed_body}
\mfk{r}_P(X) \leq \uk_k(\Omega) \quad\text{where}\quad \text{$P$ is lax and } \on{codim}(P) = 2k
\end{equation}
\end{thm}

\begin{proof} Let $\Omega \subset \R^2$ be the rational polytope. We may assume (after resolving singularities via blow up) that $\Omega$ is smooth. Let $C$ be a movable curve class in $\on{Mov}(Y_\Sigma)^+$ in $Y_\Sigma$ such that
\[\on{codim}(P) \le \on{ind}(C) = 2(c_1(C) - 1) = -2(K_\Omega \cdot C + 1)\]
Let $2m = \on{ind}(C) - \on{codim}(P)$ and let $Q$ be the tangency constraint acquired by taking the union of $P$ and $m$ points with tangency order $0$, i.e. $m$ copies of the tangency condition $((0))$. Note that $Q$ is lax since $P$ is. Thus, by Lemma \ref{lem:GW_of_movable_curve}
\[\on{GW}_Q(Y_\Sigma,C) > 0\]
Therefore, by the Monotonicity and Gromov--Witten axioms of the RSFT capacities in Theorem \ref{thm:tangent_GW_invariants}, we have
\[
\mathfrak{r}_P(X) \le \mathfrak{r}_Q(X) \le [\omega] \cdot C = A_\Omega \cdot C
\]
Minimizing over all such $+$ movable classes yields the result.
\end{proof}

Theorem \ref{thm:main_upperbound_stable} is the following corollary of the above Gromov--Witten calculation. Recall that this result only applies to the tangency constraints $((k))$, denoted by $P(k)$ in the introduction.

\begin{cor} \label{cor:main_upperbound_stable_body} Let $X \to Y_\Sigma \times Z$ be a symplectic embedding of a star-shaped domain $X$ into the product of a strongly convex toric surface $Y_\Sigma$ and a closed symplectic manifold $Z$. Then
\begin{equation} \label{eqn:main_upperbound_stable_body}
\mfk{r}_{P(k)}(X) \leq \uk_{k+1}(\Omega) 
\end{equation}
\end{cor}

\begin{proof} By the stabilization property of the Gromov--Witten invariants $\on{GW}_{((k))}$ (Lemma \ref{lem:GW_stabilization}) and Theorem \ref{thm:main_upperbound_closed_body}, we have
\[\on{GW}_P(Y_\Sigma \times Z, A \otimes [\on{pt}]) = \on{GW}(Y_\Sigma,A) > 0\]
The rest of the proof is identical to that of Theorem \ref{thm:main_upperbound_closed_body}
\end{proof}

Finally, a simple limiting argument using Theorem \ref{thm:main_upperbound_closed}/\ref{thm:main_upperbound_closed_body} demonstrates Theorem \ref{thm:main_upperbound}/\ref{thm:main_upperbound_body}.

\begin{thm} \label{thm:main_upperbound_body} Let $\Omega \subset \R^2$ be a strongly convex moment domain and let $P$ be a lax tangency constraint of codimension $2k$. Then
\[\mfk{r}_P(X_\Omega) \leq \uk_k(\Omega)\]
\end{thm}

\begin{proof} Choose a sequence $\Omega_i$ of strongly convex, smooth and rational polytopes such that
\[(1 - 1/i) \cdot \Omega_i \subset \Omega \subset \Omega_i\]
The same identity then holds for the domains corresponding to $\Omega$ and $\Omega_i$. Thus by the Scaling and Monotonicity axioms (Theorem \ref{thm:main_rP_axioms}(a)), we have
\[
\mathfrak{r}_P(X_{\Omega}) \le \mathfrak{r}_P(X_{\Omega_i}) \le \mathfrak{u}_k(\Omega_i)
\]
Since $\mathfrak{l}_k(\Omega_i) \to \mathfrak{l}_k(\Omega)$ as $i \to \infty$ by Lemma \ref{lem:lk_uk_properties}(c), we can take the limit as $i \to \infty$ to acquire the desired result.
\end{proof}

\bibliography{bw_jc}
\bibliographystyle{acm}

\end{document}